\documentclass[12pt]{amsart}
\usepackage{psfrag}
\usepackage{graphicx}
\usepackage{pinlabel}
\usepackage{graphicx}
\usepackage{amsmath}
\usepackage{amssymb}
\usepackage{amscd}
\usepackage{picinpar}
\usepackage{times}
\usepackage{pb-diagram}
\usepackage{graphicx}
\usepackage{wrapfig}
\usepackage{xspace}

\def\C{\mathbb{C}}
\def\R{\mathbb{R}}
\def\H{\mathbb{H}}
\def\T{\mathcal{T}}

\def\Z{\mathbb{Z}}

\def\D{\mathbb{D}}
\def\SS{\mathbb{S}}

\newtheorem{theorem}{Theorem}[section]
\newtheorem{lemma}[theorem]{Lemma}
\newtheorem{proposition}[theorem]{Proposition}
\newtheorem{corollary}[theorem]{Corollary}
\newtheorem{definition}[theorem]{Definition}
\newtheorem{conj}[theorem]{Conjecture}
\newtheorem{remark}[theorem]{Remark}

\begin{document}

\title[Convergence of discrete conformal maps]
%
{Discrete conformal geometry of polyhedral surfaces and its convergence}




\author{Feng Luo}

\address{Department of Mathematics, Rutgers University, Piscataway, NJ, 08854}

\email{fluo@math.rutgers.edu}

\author{Jian Sun}

\address{Mathematical Science Center, Tsinghua University, Beijing, 100084, China}

\email{jsun@math.tsinghua.edu.cn}

\author{Tianqi Wu}

\address{ Center of Mathematical Sciences and Applications, Harvard University, Cambridge, MA 02138  }

\email{tianqi@cmsa.fas.harvard.edu}



\begin{abstract}
The  paper proves a result on the  convergence of discrete conformal maps to the Riemann mappings for Jordan domains.  It is a counterpart of Rodin-Sullivan's theorem on convergence of circle packing mappings to the Riemann mapping in the new setting of discrete conformality.  The proof follows the same strategy that Rodin-Sullivan used by establishing a rigidity result for regular hexagonal triangulations of the plane and estimating the quasiconformal constants associated to the discrete conformal maps.




\end{abstract}

\maketitle

\addtocontents{toc}{\setcounter{tocdepth}{1}}
\tableofcontents



\section{Introduction}

Thurston's conjecture on the convergence of circle packing mappings to the Riemann mapping is a constructive and geometric approach to the Riemann mapping theorem. The conjecture was solved in an important work by Rodin and Sullivan \cite{RS} in 1987. There have been many research works inspired by the work of  Thurston and Rodin-Sullivan since then. This paper addresses  a counterpart of Thurston's convergence conjecture in the setting of discrete conformal change of polyhedral surfaces associated to the notion of vertex scaling (Definition \ref{dc}). We prove a weak version of Rodin-Sullivan's theorem in this new setting.  There are still many problems to be resolved in order to prove the full convergence conjecture.

Let us begin with a recall of Thurston's conjecture and Rodin-Sullivan's solution.  Given a bounded simply connected domain $\Omega$ in the complex plane $\C$, one constructs a sequence of  approximating triangulated polygonal disks $(D_n, \T_n)$ whose triangles are equilateral and edge lengths of the triangles tend to zero such that $D_n$ converges to $\Omega$. For each such polygonal disk, by the Koebe-Andreev-Thurston's existence theorem, there exists a circle packing of the unit disk $\D$ such that the combinatorics (or the nerve) of  circle packing is isomorphic to the 1-skeleton of the triangulation $\T_n$. This produces a piecewise linear homeomorphism $f_n$, called the circle packing mapping,  from the polygonal disk $D_n$ to a polygonal disk inside $\D$ associated to the circle packing. Thurston conjectured in 1985  that, under appropriate normalizations, the sequence $\{f_n\}$ converges uniformly on compact subsets of $\Omega$ to the Riemann mapping for $\Omega$.  Here the normalization condition is given by choosing a point $p \in \Omega$,  a sequence of vertices $v_n$ in $(D_n, \T_n)$ such that $\lim_n v_n =p$, and $f_n(v_n) =0$ such that $f_n'(v_n)>0$.  The Riemann mapping $f$ for $\Omega$ sends  $p$ to $0$ and $f'(p)>0$.  Rodin-Sullivan's proof of Thurston's conjecture is elegant and goes in two steps. In the first step, they show that the circle packing mappings $f_n$ are $K$-quasiconformal for some constant $K$ independent of the indices. In the second step, they show  that there is only one hexagonal circle packings of the complex plane up to Moebius transformations. This implies that the limit of the sequence $\{f_n\}$ is conformal.

Circle packing metrics introduced by Thurston \cite{th} can be considered as a discrete conformal geometry of polyhedral surfaces. In recent times, there have been  many works on discretization of 2-dimensional conformal geometry (\cite{luo}, \cite{bps}, \cite{hersonsky1}, \cite{glsw}, \cite{glickenstein}, and others). In this paper, we consider the counterpart of Thurston's conjecture in the setting of discrete conformal change defined by vertex scaling. 

To state our main results, let us recall some related material and notations.
A compact topological surface $S$ together with a
non-empty finite subset of points $V \subset S$ will be called a
\it marked surface\rm. A triangulation $\T$ of a marked surface
$(S,V)$ is a topological triangulation of $S$ such that the vertex
set of $\T$ is $V$. We use $E=E(\T)$, $V=V(\T)$ to denote the sets
of all edges and vertices in $\T$ respectively. A \it
polyhedral metric \rm $d$ on $(S,V)$, to be called a \it PL metric
\rm on $(S,V)$ for simplicity, is a flat cone metric on $(S,V)$
whose cone points are contained in $V$.  We call the triple
$(S,V,d)$ a polyhedral surface. The \it discrete curvature\rm,  or
simply \it curvature\rm, of a PL metric $d$ is the function $K: V
\to (-\infty, 2\pi)$ sending an interior vertex $v$ to $2\pi$ minus the cone
angle at $v$ and a boundary vertex $v$ to $\pi$ minus the sum of angles at $v$.
All PL metrics are obtained by isometric gluing of Euclidean
triangles along pairs of edges.  If $\T$ is a triangulation of a
polyhedral surface $(S,V, d)$  for which all edges in $\T$ are
geodesic, we say $\T$ is \it geometric \rm in $d$ and $d$ is a PL
metric on $(S, \T)$. In this case, we can represent $d$ by the
length function $l_d: E(\T) \to \R_{>0}$ sending each edge to its
length. Thus the polyhedral surface $(S, V, d)$ can be represented
by $(S, \T, l_d)$ where $l_d \in \R_{>0}^E$. We will also call
$(S, \T, l_d)$ or  $l_d$ a PL metric on $\T$.
\begin{definition} \label{dc}(Vertex scaling change of PL metrics \cite{luo})  Two PL metrics $l$ and $l^*$ on a triangulated surface
$(S, \T)$ are related by a \emph{vertex scaling} if there exists a map $w: V(\T) \to \R$ so that if $e$ is an edge in $\T$ with end points
$v$ and $v'$, then the edge lengths $l(e)$ and $l^*(e)$
are related by
\begin{equation}\label{conformalll}
l^*(e) = e^{ w(v)+w(v')} l(e).
\end{equation}  We denote $l^*$ by $w*l$ if (\ref{conformalll}) holds and  call $l^*$  obtained from $l$ by a vertex scaling and $w$ a discrete conformal factor.
\end{definition}
Condition (\ref{conformalll}) was proposed in \cite{luo} as a discrete conformal equivalence between PL metrics on triangulated surfaces. There are three basic problems related to the vertex scaling. The first is the existence problem. Namely, given a PL metric $l$ on a triangulated closed surface $(S, \T)$ and a function $K: V(\T) \to (-\infty, 2\pi)$ satisfying the Gauss-Bonnet condition, is there a PL metric $l^*$ of the form $w*l$ whose curvature is $K$?  Unlike Koebe-Andreev-Thurston's theorem which guarantees the existence of circle packing metrics, the answer to the above existence problem is negative in general. This makes the convergence of discrete conformal mappings a difficult problem.  On the other hand,  the uniqueness of the vertex scaled PL metric  $l^*$ with prescribed curvature holds. This was  established in an important paper by Bobenko-Pinkall-Springborn \cite{bps}.  The third is the convergence problem. Namely, assuming  the existence of PL metrics with prescribed curvatures, can these discrete conformal polyhedral surfaces approximate a given Riemann surface?  The main result of the paper gives a solution to the convergence problem for the simplest case of Jordan domain.


The convergence theorem that we proved is the following.
Let $\Omega$ be a Jordan domain with three points $p,q,r$ specified in the boundary.
By  Caratheodory's extension theorem
\cite{pomo},  the Riemann mapping
from  $\Omega$ to the unit disk $\D$  extends to a homeomorphism from the closure
$\overline{\Omega}$ to the closure $\overline{\D}$. Therefore, there
exists a unique homeomorphism $g$ from $\overline{\Omega}$  to  an equilateral Euclidean triangle $\Delta ABC$ with vertices $A,B,C$  such that  $p,q,r$ are sent to $A,B,C$  and $g$ is conformal in $\Omega$. For simplicity, we call $g$ and  $g^{-1}$
the  Riemann mappings for $(\Omega, (p,q,r))$.

Given an oriented triangulated polygonal disk $(D, \T, l)$ and
three boundary vertices $p,q,r \in V$,  suppose there exists a  PL metric $l^*=w*l$ on $(D, \T)$ for some $w: V \to \R$ such that
its discrete curvature at all vertices except $\{p, q, r\}$ are zero and the curvatures at $p, q, r$ are $\frac{2\pi}{3}$.  Then the associated flat metric on $(D, \T, l^*)$ is isometric to an equilateral triangle $\Delta ABC$, i.e., there is a geometric triangulation $\T'$ of $\Delta ABC$ such that $(\Delta ABC, \T', l_{st})$ is isometric to $(D, \T, l^*)$.  Here   and below, if $\T$ is a geometric triangulation of a domain in the plane,  we use $l_{st}: E(\T) \to \R$ to denote the length of edges $e$ in $\T$ in the standard metric on $\C$.
 Let $f: D \to \Delta ABC$ be the piecewise linear orientation preserving homeomorphism sending $V$ to the vertex set $V(\T')$ of $\T'$, and $p, q, r$ to $A, B, C$ respectively  and being  linear on each triangle of $\T$.  We call $f$ the \it discrete uniformization map \rm associated to
$(D, \T, l, \{p,q,r\})$.  Note that $f$ may not exist due to the lacking of existence theorem.

%
%

\begin{theorem}\label{conv} Suppose $\Omega$ is a Jordan domain in the complex
plane with three distinct points $p,q,r \subset \partial \Omega$.
 Then there exists a sequence $(\Omega_{n},
\T_n, l_{st}, (p_n, q_n, r_n))$ of simply connected triangulated
polygonal disks in $\C$
 where $\T_n$ are
triangulations by equilateral triangles and $p_n, q_n, r_n$ are
three boundary vertices such that

(a) $\Omega=\cup_{n=1}^{\infty} \Omega_n$ with $\Omega_n \subset
\Omega_{n+1}$,  and  $\lim_n p_n =p$, $\lim_n q_n =q$ and $\lim_n r_n=r$,

(b)  discrete uniformization maps associated to $(\Omega_{n},
\T_n, l_{st}, (p_n, q_n, r_n))$  exist and converge uniformly to the
 Riemann mapping for
 $(\Omega, (p,q,r))$.
\end{theorem}

In Rodin-Sullivan's convergence theorem, any sequence of
approximating circle packing maps associated to the approximation
triangulated polyhedral disks $\Omega_n$ such that $\Omega_n \subset
int(\Omega_{n+1})$ and $\Omega =\cup_{n} \Omega_n$ converges to
the Riemann mapping. Theorem \ref{conv} is less robust in this
aspect since  discrete conformal maps may not exist if the triangulations $\T_n$ are not carefully selected.
A stronger version of convergence is conjectured in \S7.
The conformality of the limit of the discrete conformal maps in Theorem
\ref{conv} is a consequence of 
the following result.  
Recall that a geometric triangulation $\T$ of polyhedral surface is called \it Delaunay \rm if the
sum of two angles facing each interior edge is at most $\pi$. Delaunay
triangulations always exist for each PL metric on compact surfaces.

\begin{theorem} \label{rigid}  Suppose $\T$ is a Delaunay geometric triangulations of the complex plane $\C$ such that its vertex set is a lattice and $l_{st}: E(\T) \to \R$ is the edge length function of $\T$.  If  $(\C, \T, w*l_{st})$ is a Delaunay triangulated surface isometric to an open set in the Euclidean plane $\C$,  then $w$ is a constant function.  \rm
\end{theorem}

We remark that the same result as above for the standard hexagonal lattice has been proved independently by Dai-Ge-Ma \cite{dai-ge} in a recent preprint.

Using an important result in \cite{bps} that vertex scaling is closely related to hyperbolic 3-dimensional geometry and the work of \cite{glsw}, one sees that Theorem \ref{rigid} implies the following rigidity result on convex hyperbolic polyhedra.

\begin{theorem} \label{rigid}  Suppose $L =\Z+\tau \Z$ is a lattice in
the plane $\C$ and $V \subset \C$ is a discrete set such that
there exists an isometry between the boundaries of the convex
hulls of $L$ and $V$ in the hyperbolic 3-space $\H^3$ preserving
cell structures. Then $V$ and $L$ differ by a complex affine
transformation of $\C$.
\end{theorem}

This prompts us to propose the following conjecture. A closed set
$X$ in the Riemann sphere is said to be of \it circle type \rm if
each connected component of $X$ is either a point or a round disk.
Consider the Riemann sphere $\C \cup\{\infty\}$ as the infinity of
the (upper-half-space model of) hyperbolic 3-space $\H^3$.

\begin{conj}\label{c1}
For any genus zero connected complete
hyperbolic surface $\Omega$, there exists a circle type closed set
$X \subset \C\cup \{\infty\}$ such that $\Omega$ is isometric to
the boundary of the convex hull of $X$ in $\H^3$. \end{conj}

\begin{conj}\label{c2} Suppose $X$ and $Y$ are circle type closed sets in $\C$ such that
boundary of the convex hulls of $X$ and $Y$ in $\H^3$ are isometric. Then $X$ and $Y$ differ by a M\"obius transformation.
\end{conj}




The paper is organized as follows. \S2 recalls the basic material
for discrete conformal geometry of polyhedral surfaces. Sections
3 and 4 are devoted to prove Theorem \ref{rigid}. The main tools
used are a maximum principle, a variational principle for discrete
conformal geometry of polyhedral surfaces and spiral hexagonal
triangulations derived from linear conformal factors. Section \S5
investigates the existence of flat metrics with prescribed boundary curvature on   polygonal disks.
The main result (Theorem 5.1) is an  existence result for vertex scaling equivalence if triangulations of a polyhedral disk are sufficiently fine subdivided.
The basic tools used are discrete harmonic functions, their
gradient estimates and solutions to ordinary differential
equations.  We prove the convergence Theorem \ref{conv} in \S6
using the results obtained in \S4, \S5 and Rado-Palka's theorem on uniform convergence of Riemann mappings and
quasiconformal mappings. \S7 discusses a strong version of the convergence of discrete uniformization maps and the
 motivations for Conjecture
\ref{c1}.

\bigskip
\noindent {\bf Acknowledgement}. We thank Michael Freedman for
discussions which lead to the formulation of Conjectures \ref{c1}, \ref{c2}.   The
work is partially supported by the NSF DMS 1405106,  NSF DMS 1760527, NSF DMS 1811878, NSF DMS 1760471  of USA and a grant from  the NSF of China.


\section{Polyhedral metrics,
vertex scaling and a variational principle}

We begin with some notations. Let $\C$, $\R$, $\Z$ be the sets of
complex, real, and integers respectively. $\R_{>0}=\{ t \in \R|
t>0\}$, $\Z_{\geq k}=\{ n \in \Z| n \geq k\}$ and $\SS^1=\{ z \in
\C| |z|=1\}$. We use $\D$ to denote the open unit disk in $\C$ and
$\H^n$ to denote the $n$-dimensional hyperbolic space.

Given that $X$ is a compact surface with boundary,
its interior is denoted by $int(X)$. A graph with
vertex set $V$ and edge set $E$ is denoted by $(V, E)$. Two
vertices $i, j$ in a graph $(V,E)$ are \it adjacent\rm, denoted
 by $i \sim j$, if they are the end points of an edge. If $i \sim j$,
 we use $[ij]$ (respectively $ij$) to denote an oriented (respectively unoriented) edge from $i$ to $j$.
 An \it edge path \rm joining $i,j \in V$ is
  a sequence of vertices
$\{v_0=i, v_1, ..., v_m=j\}$ such that $v_k \sim v_{k+1}$. The
length of the path is $m$. The \it combinatorial distance \rm
$d_c(i,j)$ between two vertices in a connected graph $(V, E)$ is
the length of the shortest edge path joining $i,j$. Suppose $(S,
\T)$ is a triangulated  surface with possibly non-empty boundary
$\partial S$ and possibly  non-compact $S$. Let $E=E(\T)$,
$V=V(\T)$ be the sets of edges, vertices respectively and
$\T^{(1)}=(V, E)$ be the associated graph.    A vertex $v \in
V(\T) \cap
\partial S$ (resp. $v \in V\cap (S -\partial S))$ is called a
boundary (resp. interior) vertex. Boundary and interior edges are
defined in the same way. A PL metric on  $(S, \T)$ or simply on
$\T$ can be represented by a length function $l: E(\T) \to
\R_{>0}$ so that if $e_i, e_j, e_k$ are three edges forming a
triangle in $\T$, then the \it strict triangle inequality \rm
holds,
\begin{equation}\label{striangle}
l(e_i)+l(e_j) > l(e_k).
\end{equation}

We will use limits of PL metrics.  To this end, we
introduce the notion of \it generalized PL metrics \rm on $(S,
\T)$. Take three pairwise distinct points $v_1, v_2, v_3$ in the
plane. The convex hull of $\{v_1, v_2, v_3\}$ is a \it generalized
triangle \rm with vertices $v_1, v_2, v_3$. We denoted it by  $\Delta v_1v_2v_3$. If
$v_1, v_2, v_3$ are not in a line, then $\Delta v_1v_2v_3$ is a
(usual) triangle. If $v_1,v_2,v_3$ lie in a line, then $\Delta
v_1v_2v_3$ is a \it degenerate triangle \rm with the flat vertex
at $v_i$ if $|v_j-v_i|+|v_k-v_i|=|v_j-v_k|$,
$\{i,j,k\}=\{1,2,3\}$. Let
 $l_i=|v_j-v_k| \in \R_{>0}$ be the \it
edge length \rm and $a_i \in [0, \pi]$ be the \it angle \rm at
$v_i$. Then $l_i+l_j \geq l_k>0$  and the angles are given by
\begin{equation}\label{cosine}
a_i =\arccos(\frac{l_j^2+l_k^2-l_i^2}{2 l_j l_k}).
\end{equation}
Furthermore, the angle $a_i=a_i(l_1, l_2,l_3)\in [0, \pi]$ is
continuous in $(l_1, l_2, l_3)$. Degenerate triangles are
characterized by either having an angle $\pi$ or the lengths
satisfying $l_i=l_j+l_k$ for some $i,j,k$.

A \it generalized PL metric \rm on a triangulated surface $(S, T)$
is represented by an edge length function  $l: E(\T) \to \R_{>0}$
so that if $e_i, e_j, e_k$ are three edges forming a triangle in
$\T$, then the  triangle inequality holds,
\begin{equation}\label{triangle}
l(e_i)+l(e_j) \geq l(e_k).
\end{equation}
We will abuse the use of terminology and call $l$ a generalized PL
metric on $(S, \T)$ or $\T$.
 The
 \it discrete curvature \rm $K: V(\T) \to (-\infty, 2\pi]$ of a
generalized PL metric  $(S,\T, l)$ is defined as follows.
 If $v \in V(\T)$ is an interior
vertex, $K(v)$ is $2\pi$ minus the sum of angles (of generalized
triangles) at $v$; if $v$ is a boundary vertex, $K(v)$ is $\pi$
minus the sum of angles at $v$. Note that the Gauss-Bonnet theorem
$\sum_{v \in V(\T)} K(v)=2\pi\chi(S)$ still holds for a compact
surface $S$ with a generalized PL metric. Clearly the curvature
$K$ and inner angles depend continuously on the length vector
$l\in \R_{>0}^{E(\T)}$. A generalized PL metric is called \it flat
\rm if its curvatures are zero at all interior vertices $v$. A
generalized PL metric $(S, \T, l)$ (or sometimes written as $(\T,
l)$) is called \it Delaunay \rm if for each \it interior \rm edge
$e \in E(\T)$ the sum of the two angles facing $e$ is at most
$\pi$. If $(S, \T, l)$ is a Delaunay generalized PL metric such
that each angle facing a boundary edge is at most $\pi/2$, then
the metric double of $(S, \T, l)$ along its boundary is a Delaunay
triangulated generalized PL metric surface. Two generalized PL
metrics $l$ and $\tilde{l}$ on $(S, \T)$ are related by a \it
vertex scaling \rm if there is $w \in \R^V$ so that
$$ \tilde{l}(vv') =e^{ w(v)+w(v')}l(vv')$$ for all edges $vv' \in E(\T)$.
 We write $\tilde{l} = w*l$ and call $w$ a \it discrete conformal
 factor\rm.

Two generalized triangles $\Delta v_1v_2v_3$ and $\Delta
u_1u_2u_3$ are \it equivalent \rm if there exists an isometry
sending $v_i$ to $u_i$ for $i=1,2,3$. The space of all equivalence
classes of generalized triangles can be identified with $\{ (l_1,
l_2, l_3) \in \R_{>0}^3 | l_i + l_j\geq l_k\}$. It contains the
space of all equivalence classes of triangles $\{ (l_1, l_2, l_3)
\in \R_{>0}^3 | l_i + l_j > l_k\}$. Given two generalized
triangles $l=(l_1, l_2, l_3)$ and $\tilde{l}=(\tilde{l_1},
\tilde{l_2}, \tilde{l_3})$ there exists $w=(w_1, w_2, w_3) \in
\R^3$ such that $\tilde{l_i} =l_i e^{w_j+w_k}.$

The following result was proved in \cite[Theorem 2.1]{luo} for
Euclidean triangles. The extension to generalized triangles is
straight forward.

\begin{proposition}[\cite{luo}] \label{variation} Let $\Delta v_1v_2v_3$ be a
 fixed generalized
triangle of edge length vector $l=(l_1, l_2, l_3)$ and $w*l$ is the edge length vector of a
vertex scaled generalized triangle whose inner angle at $v_i$ is
$a_i=a_i(w)$.  


\noindent (a) For any two constants $c_i, c_j$, the set $\{(w_1,
w_2, w_3) \in \R^3| w*l$ is a generalized triangle and $w_i=c_i$,
$w_j=c_j$\} is either connected or empty.

\noindent (b) If $(\Delta v_1v_2v_3, l)$ is a non-degenerate
triangle and $i,j,k$ distinct,
 then
\begin{equation} \label{12}
\frac{\partial a_i}{\partial w_i}|_{w=0}
=-\frac{\sin(a_i)}{\sin(a_j)\sin(a_k)}<0, \quad \frac{\partial
a_i}{\partial w_j}|_{w=0} =  \frac{\partial a_j}{\partial
w_i}|_{w=0} =\cot(a_k), \quad \text{and} \quad \sum_{j=1}^3
\frac{\partial a_i}{\partial w_j}=0.
\end{equation}
The matrix $-[\frac{\partial a_r}{\partial w_s}]_{3 \times 3}$ is
symmetric, positive semi-definite with null space spanned by
$(1,1,1)^T$.

\noindent (c) If  $(\Delta v_1v_2v_3, l)$  is a degenerate
triangle having $v_3$ as the flat vertex,  then for small $t>0$,
$(\Delta v_1v_2v_3, (0,0,t)*l)$ is a non-degenerate triangle. The
angle $a_3(0,0,t)$ is strictly decreasing in $t$ for all $t$ for
which $(0,0,t)*l$ is a generalized triangle. The angles $a_i(0,0,t
)$, $i=1,2$, are strictly increasing in $t \in [0, \epsilon)$ for
some $\epsilon >0$.



\end{proposition}

\begin{figure}[ht!]
\begin{center}
\begin{tabular}{c}
\includegraphics[width=0.66\textwidth]{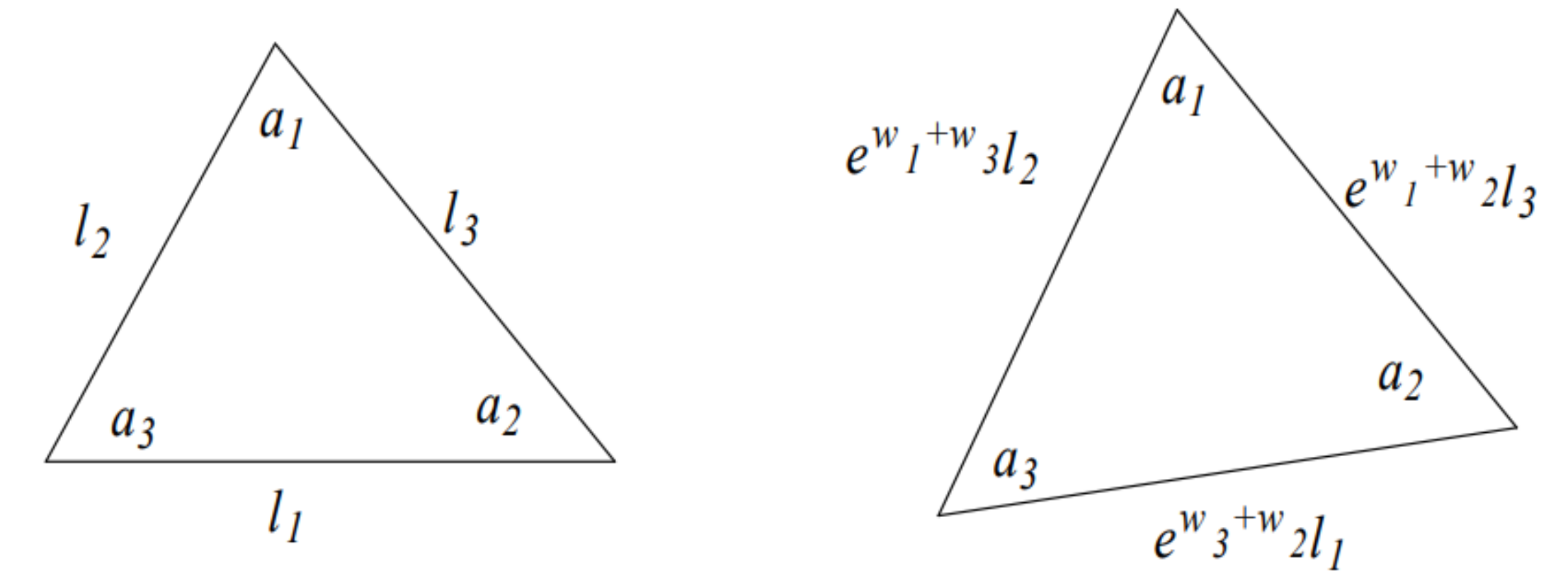}
\end{tabular}
\end{center}
\caption{Vertex scaling of a triangle} \label{555}
\end{figure}


\begin{proof}  To see part (a), without loss of generality,  we may assume $c_1$
and $c_2$ are the given constants. Then the variable $w_3$ is
defined by three inequalities $e^{w_3}( e^{c_1}l_2 +e^{c_2}l_1)
\geq e^{c_1+c_2}l_3$, $ e^{c_1+c_2}l_3  \geq e^{w_3}( e^{c_1}l_2
-e^{c_2}l_1) \geq -e^{c_1+c_2}l_3$. Each of these inequalities
defines an interval in $w_3$ variable.  Therefore the solution
space is either the empty set or a connected set.

Part (b) is in \cite[Theorem 2.1]{luo}.

 To see (c), due to $l_3=l_1+l_2$, for small $t>0$,
we have $(0,0,t)*l=(e^tl_1, e^tl_2, l_1+l_2) \in \mathbf \Delta  :=
\{(x_1, x_2, x_3) \in \R^2_{>0} | x_i +x_j \geq x_k\}$.  Now by (\ref{12}) and the Sine Law, $\frac{\partial a_3}{\partial
w_3}(0,0,t)=-\frac{\sin(a_3)}{\sin(a_1)\sin(a_2)} <0$.
Together with part (a), the angle $a_3(0,0,t)$ as a function of
$t$ is defined on an interval and is strictly decreasing in $t$.
Since $\lim_{t \to 0^+} \frac{\partial a_1}{\partial w_3} =\lim_{t
\to 0^+} \cot(a_2) = \infty$ due to $\lim_{t \to 0^+}
a_2(t,0,0)=0$, therefore the result holds for $a_1$. By the same
argument, the result holds for $ a_2$.


\end{proof}


As a consequence,
\begin{corollary}\label{trig} Under the same assumption as
in Proposition \ref{variation}, if $w(t)=(w_1(t), w_2(t), w_3(t))
\in \R^3$ is smooth in $t$ such that $w(t)*l$ is the edge length
vector of a triangle with inner angle
$a_i(t)=a_i(w(t)*l)$ at $v_i$, then
\begin{equation} \label{893} \frac{d a_i(t)}{dt} =\sum_{j \sim i}
\cot(a_k)(\frac{d w_j}{dt} -\frac{d w_i}{dt})
\end{equation}
where  $j \sim i$ means $v_j$ is adjacent to $v_i$ and
$\{i,j,k\}=\{1,2,3\}$.
\end{corollary}

Write $w'_j(t)=\frac{dw_j}{dt}$. Indeed by the chain rule and
(\ref{12}), we have
$$\frac{d a_i(t)}{dt} =\frac{\partial a_i}{\partial w_i}w_i'+
\sum_{j \neq i} \frac{\partial a_i}{\partial w_j} w_j'  $$
$$= -\sum_{j \neq i} \cot(a_k) w_i'+ \sum_{j \neq i}
\cot(a_k)w_j'  = \sum_{j \sim i} \cot(a_k)(w_j'-w_i').$$

Suppose $(S, \T, l)$ is a geometrically triangulated compact
polyhedral surface and $w(t) \in \R^V$ is a smooth  path in parameter $t$ such that $w(t)*l$ is a PL metric on ($S, \T$). Let $K_i
=K_i(t)$ be the discrete curvature  at $i\in V$ and
$\theta^i_{jk}=\theta^i_{jk}(t)$ be the inner angle at the vertex
$i$ in $\Delta ijk$ in the metric $w(t)*l$.  For an edge $[ij]$ in
the triangulation $\T$,  define $\eta_{ij}$ to be
$\cot(\theta^k_{ij})+\cot(\theta^l_{ij})$ if $[ij]$ is an interior
edge facing two angles $\theta^k_{ij}$ and $\theta^l_{ij}$ and
$\eta_{ij}=\cot(\theta^k_{ij})$ if $[ij]$ is a boundary edge.  If
$[ij]$ is an interior edge, then $\eta_{ij} \geq 0$ if and only if
$\theta^k_{ij}+\theta^l_{ij} \leq \pi$, i.e., the Delaunay
condition holds at $[ij]$.

The curvature variation formula is the following,
\begin{proposition}
\begin{equation}\label{curvatureevol}
\frac{d K_i(t)}{dt} =\sum_{ j \sim i} \eta_{ij}(\frac{d w_i}{dt}
-\frac{d w_j}{dt}).
\end{equation} \end{proposition}

This follows directly from the Corollary \ref{trig} since $K_i=c
\pi-\sum_{r, s \in V} \theta^i_{rs}$ where $c=1$ or $2$ and
$\theta^i_{rs}$ are angles at $i$. Since $\frac{d K_i(t)}{dt}
=-\sum_{r,s \in V} \frac{d \theta^i_{rs}}{dt}$,
(\ref{curvatureevol}) follows from (\ref{893}) and the definition
of $\eta_{ij}$.

\section{A maximum principle, a ratio lemma and spiral hexagonal triangulations}

Let $v_0$ be an interior point of a star-shaped $n$-sided polygon
$P_n$ having vertices $v_1, ..., v_n$ labelled cyclically. The
triangulation $\T$ of $P_n$ with vertices $v_0, ..., v_n$ and
triangles $\Delta v_0v_iv_{i+1}$ ($v_{n+1}=v_1$) is called a \it
star triangulation \rm of $P_n$. See Figure \ref{f111}.

\begin{theorem}[Maximum principle] \label{max}  Let $\T$ be
a star triangulation of $P_n$  and $l: E(\T) \to \R_{>0}$ be  a
generalized Delaunay polyhedral metric on $\T$. If $w:\{v_0,
v_1,..., v_n\} \to \R$ satisfies


(a) $w*l$ is a generalized Delaunay polyhedral metric on $\T$,

(b) the curvatures $K_0(w*l)$ of $w*l$ and $K_0(l)$ of $l$ at
vertex $v_0$  satisfy $K_{0}(w*l) \leq K_{0}(l)$, and

(c) $w(v_0)=\max\{ w(v_i)| i=0,1,..., n\}$,

then $w(v_i)=w(v_0)$ for all $i$.
\end{theorem}

As a convention, if $x=(x_0, ..., x_m)$ and $y=(y_0, y_1, ...,
y_m)$ are in $\R^{m+1}$, then $x \geq y$ means $x_i \geq y_i$ for
all $i$. Given $w:\{v_0, ..., v_m\} \to \R$, we use $w_i =w(v_i)$
and identify $w$ with $(w_0, w_1, ..., w_m) \in \R^{m+1}$. The
cone angle of $w*l$ at $v_0$ will be denoted by $\alpha(w)$. Thus
 Theorem \ref{max}(b) says $\alpha(w) \geq \alpha(0)$.

The proof of Theorem \ref{max} depends on the following lemma.

\begin{lemma} \label{key} If $w:\{v_0,
v_1,..., v_n\} \to \R$ satisfies (a), (b), (c) in Theorem
\ref{max} such that there is $w_{i_0} < w_0$, then there exists
$\hat{w} \in \R^{n+1}$ such that

(a) $\hat{w}_i \geq w_i$ for $i=1,2,..., n$,

(b) $\hat{w}_i \leq  \hat{w}_0=w_0$ for $i=1,2,...,n$,

(c)  $\hat{w}*l$ is a generalized Delaunay polyhedral metric on
$\T$, and

(d)  \begin{equation} \alpha(\hat{w})
> \alpha(w). \end{equation}
\end{lemma}

Let us first prove  Theorem \ref{max} using Lemma \ref{key}.

\begin{proof} By replacing $w$
by $w-w(v_0)(1,1,...,1)$, we may assume that $w(v_0)=0$. Suppose
the result does not hold, i.e., there exists $w$ so that $w_0=0$,
$w_i \leq 0$ for $i=1,2,..., n$ with one $w_{i_0}<0$, and $w*l$ is
a generalized Delaunay PL metric on $\T$ so that
 $\alpha(w) \geq \alpha(0)$. We will derive a contradiction as
 follows. By Lemma \ref{key}, we may assume, after replacing $w$ by
$\hat{w}$, that
\begin{equation}\label{larger} \alpha(w) > \alpha(0).
\end{equation}
 Consider the set
 $$X =\{ x \in \R^{n+1} | w \leq x \leq 0, x_0=0, \text{
 $x*l$ is a generalized Delaunay polyhedral metric on $\T$}\}.$$
 Clearly $w \in X$ and therefore $X
\neq \emptyset$ and $X$ is bounded. Since inner angles are
continuous in edge lengths, we see that $X$ is a closed set in
$\R^{n+1}$. Therefore $X$ is compact. Let $t \in X$ be a maximum
point of the continuous function $f(x)=\alpha(x)$ on $X$. We claim
that $t=0$.
To prove this, we assume $t\neq0$ and $t\leq0$. Then
by
Lemma \ref{key}, we can find $\hat{t} \geq t$ such that  $\hat{t}_0=0$ and
$\hat{t} \leq 0$, and $\hat{t}*l$ is a generalized Delaunay
polyhedral metric on $\T$ with $\alpha(\hat{t}) > \alpha(t)$. This
contradicts the maximality of $t$. 
  Now for $t=0$, we
have
$$ \alpha(0)=\alpha(t) \geq \alpha(w) > \alpha(0)$$
where the last inequality follows from (\ref{larger}). This is a
contradiction.
\end{proof}


Now back to the proof of Lemma \ref{key}. \begin{proof}  After
replacing $w$ by $w-w_0(1,1,...,1)$, we way assume $w_0=0$. Let
$a_i=a_i(w)=a_i(w_0, w_i, w_{i+1})$, $b_i=b_i(w)=b_i(w_0, w_{i-1},
w_i)$ and $c_i=c_i(w)=c_i(w_0, w_i, w_{i+1})$ be the inner angles
$\angle v_0v_{i+1}v_{i}$, $\angle v_0v_{i-1}v_i$ and $\angle
v_iv_0v_{i+1}$ in the metric $w*l$ respectively. See Figure
\ref{f111}. Let $l_i=l(v_0v_i)$ and $l_{i, i+1}
=l(v_iv_{i+1})$ be the edge lengths in the metric $l$.
\begin{figure}[ht!]
\begin{center}
\begin{tabular}{c}
\includegraphics[width=0.5\textwidth]{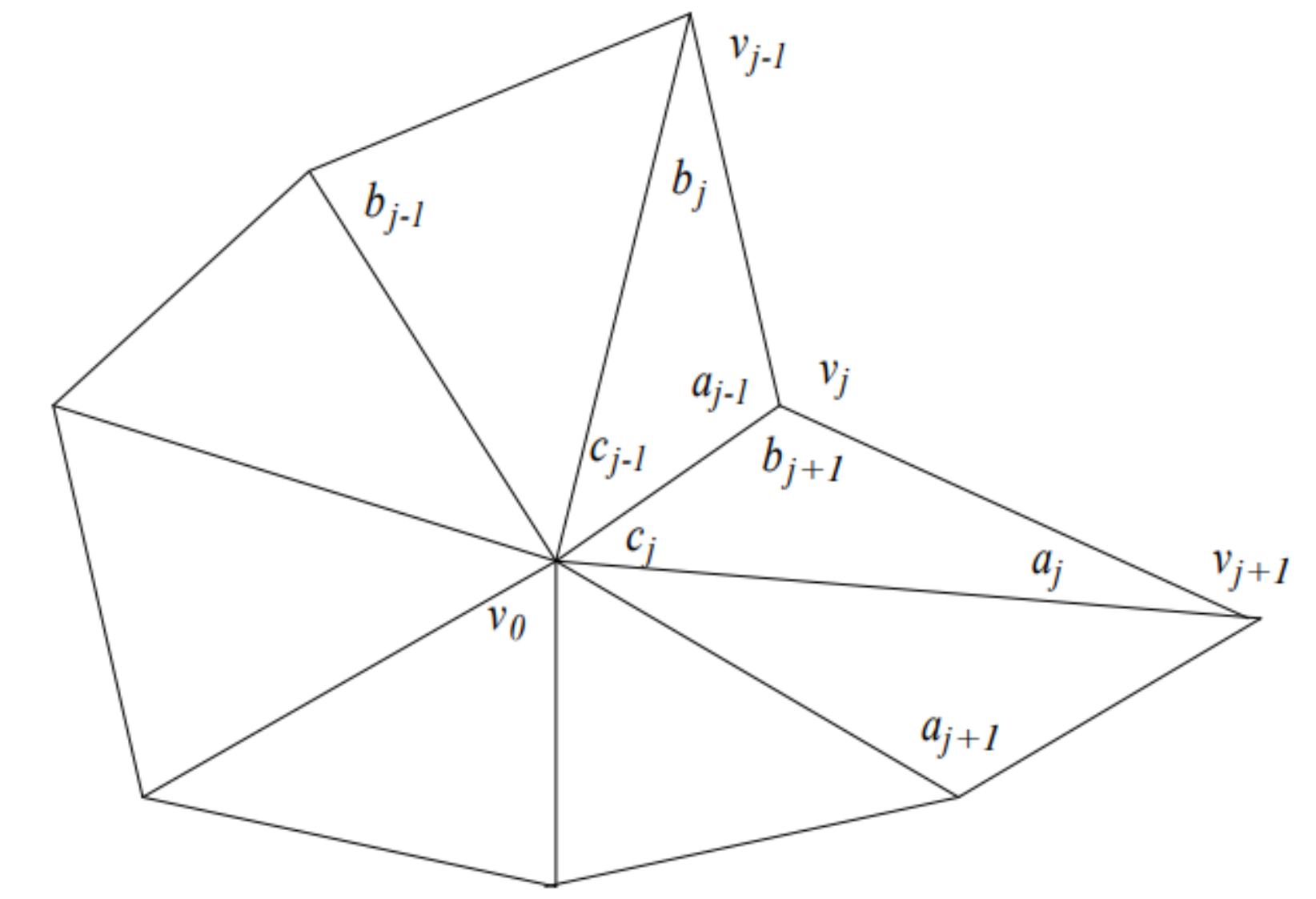}
\end{tabular}
\end{center}
\caption{Star triangulation of an $n$-sided polygon}
\label{f111}
\end{figure}

Let us begin the proof for the simplest case where all triangles
in $w*l$ are non-degenerate  (i.e., $w*l$ is a PL metric) and $w_i
<0$ for all $i \geq 1$.
Let $j\in \{1,2,3,..., n\}$ be the index such that
$w*l(v_0v_j)=\min\{ w*l(v_0v_k)| k=1,2,..., n\}$. It is well known
that in a Euclidean triangle $\triangle ABC$, $\angle A<\pi/2$ if $\overline {BC}$ is not the unique largest
edge.  Hence, due to $w*l(v_0v_j)\leq
w*l(v_0v_{j\pm 1})$, in the triangles $\Delta v_0v_jv_{j\pm 1}$,
we have
\begin{equation}\label{pi/2} \text{$a_j(w)< \pi/2$,
\quad $b_j(w) <\pi/2$ \quad and \quad $a_j(w)+b_j(w)<\pi$.}
\end{equation}

   Now consider $\hat{w}=(w_0, w_1, ...,
 w_{j-1}, w_j+t, w_{j+1}, ..., w_n)$.
 For small $t>0$, $\hat{w(t)}*l$ is still a PL metric since $w*l$ is.
 We claim $\hat{w}*l$ is still Delaunay for small $t$.
 Indeed,  by  Proposition \ref{variation}, both angles $a_{j-1}$ and
 $b_{j+1}$ decrease in $t$. On the other hand $a_{j+1}(w)=a_{j+1}(\hat{w})$ and
 $b_{j-1}(w)=b_{j-1}(\hat{w})$. Therefore, the Delaunay conditions
 $b_{j-1}+a_{j-1} \leq \pi$ and $b_{j+1}+a_{j+1}\leq \pi$ hold  for edges
 $v_0v_{j \pm 1}$.  The Delaunay condition on the edge
 $v_0v_j$ follows from choice of $j$ that $a_j+b_j<\pi$. Finally,  by Proposition
 \ref{variation}(b), $\frac{d \alpha(\hat{w})}{dt} =\cot(a_j)+\cot(b_j) =
 \frac{\sin(a_j+b_j)}{\sin(a_j)\sin(b_j)}
 >0$.   Therefore,  for small $t>0$,  we have $\alpha(\hat{w}) > \alpha(w)$.

In the general case, the above arguments  still work.

Let $J=\{j \in V | w_j <0\}$. By assumption, $J \neq \emptyset$.

{\bf Claim 1.} If $j \in J$, then $c_j(w) <\pi$ and $c_{j-1}(w) <
\pi$.

We prove $c_{j-1}(w) < \pi$ by contradiction. Suppose otherwise
that $c_{j-1}(w)=\pi$. Then the triangle $\Delta v_0v_jv_{j-1}$ is
degenerate  in $w*l$ metric, i.e., $ e^{w_j+w_{j-1}} l_{j, j-1} =
e^{w_j}l_j +e^{w_{j-1}} l_{j-1}$. Due to $w_j <0$ and $w_{j-1}
\leq 0$, we have
$$ e^{w_j+w_{j-1}} l_{j, j-1} = e^{w_j}l_j +e^{w_{j-1}} l_{j-1}
> e^{w_j+w_{j-1}}l_j +e^{w_j+w_{j-1}} l_{j-1}
=e^{w_j+w_{j-1}}(l_j+l_{j-1}).$$ This shows $l_{j, j-1} >
l_j+l_{j-1}$ which contradicts the triangle inequality for $l$
metric.  Therefore $c_{j-1}(w) <\pi$. By the same argument, we
have $c_j(w)<\pi$.  This proves Claim 1.

Let $I=\{ i>0 |w_i=0\}$ and
$$ \beta(w)=\sum_{i \in I} (b_i(w)+a_i(w))$$ and
$$\gamma(w)=\sum_{j \in J} (b_j(w) + a_j(w)).$$  Note that the
cone angle at $v_0$ is
$$\alpha(w)=\sum_{i=1}^n (\pi-a_i(w)-b_{i+1}(w))=\pi n
-\beta(w)-\gamma(w).$$ By the assumption that $\alpha(w) \geq
\alpha(0)$, we have
\begin{equation}\label{123}\beta(w)+\gamma(w) \leq
\beta(0)+\gamma(0). \end{equation}

{\bf Claim 2.} If $I \neq \emptyset$, then $\beta(w) > \beta(0).$

Indeed, if $i \in I$, i.e., $w_i=0$, then in the triangle $\Delta
v_0v_iv_{i\pm 1}$, we have $w_0=0, w_i=0$ and $w_{i\pm 1} \leq 0$.
Since $\Delta v_0v_iv_{i-1}$ are generalized triangles in both $l$
and $w*l$  metrics, by proposition \ref{variation}(a),
 we see that $\Delta v_0v_iv_{i-1}$ is a generalized triangle in
 $(w_0,..., w_{i-2}, tw_{i-1}, w_i, ..., w_n)*l$  for $t \in [0, 1]$.
By Proposition  \ref{variation} and $w_{i-1} \leq 0$,
$b_i(w_0,..., w_{i-2}, tw_{i-1}, w_i, ..., w_n)$ is  increasing in
$t \geq 0$ and is strictly increasing in $t \geq 0$ if
$w_{i-1}<0$. Therefore, $$ b_i(w)=b_i(w_0, w_{i-1}, w_i)  \geq
b_i(w_0, 0, w_i) =b_i(w_0, w_1, ..., w_{i-2}, 0, w_i, ...,
w_n)=b_i(0),$$ and $b_i(w) > b_i(0)$ if $w_{i-1} <0$. Apply the
same argument to $\Delta v_0v_iv_{i+1}$ and $a_i$, we have $a_i(w)
\geq a_i(0)$ and $a_i(w)> a_i(0)$ if $w_{i+1}<0$. Therefore
$\beta(w) \geq \beta(0)$. On the other hand, since $J \neq
\emptyset$, there exists an $i \in I$ so that either $i-1$ or
$i+1$ is in $J$. Say $i-1\in J$, i.e., $w_{i-1}<0$. Then we have
$b_i(w) > b_i(0)$ and $\beta(w)> \beta(0)$.

By claim 2 and (\ref{123}), if $I \neq \emptyset$, we conclude
that
\begin{equation}\label{1234}
 \gamma(w) =\sum_{j \in J} (a_j(w)+b_j(w)) < \gamma(0).
\end{equation}

Since $w*l$ and $l$ are Delaunay, we have $a_i(w)+b_i(w) \leq \pi$
and $a_i(0)+b_i(0)\leq \pi$ for all $i=1,2,..., n$. This implies,
by (\ref{1234}), that there exists $j \in J$ so that
\begin{equation}\label{1235}
a_j(w)+b_j(w) < \pi.
\end{equation}

If $I=\emptyset$, let $j\in J=\{1,2,3,..., n\}$ be the index such
that $w*l(v_0v_j)=\min\{ w*l(v_0v_k)| k=1,2,..., n\}$. Then the
same argument used in showing (\ref{pi/2}) and Claim 1 imply
(\ref{1235}) still holds. (Here Claim 1 is used to show that
$(w_0, ..., w_{j-1}, w_j+t, w_{j+1}, ..., w_n)*l$ is a generalized
PL metric for small $t>0$.)



Fix this $j \in J$ as above. To finish the proof, we will show
that there exists a small $t>0$ so that for $\hat{w}$=$(w_0,
w_1$,$...$, $w_{j-1}$,$ w_j+t, w_{j+1}, ..., w_n) \in \R_{\leq
0}^{n+1}$ the following hold:

(i) $\hat{w}*l$ is a generalized polyhedral metric on $\T$;

(ii) $\hat{w}*l$ satisfies the Delaunay condition;

(iii) $\alpha(\hat{w}) > \alpha(w)$.

Since $w_j<0$, any $ t \in (0, -w_j)$ will make $\hat{w} \in
\R_{\leq 0}^{n+1}$.

To see part (i), by Claim 1 and (\ref{1235}) which imply $a_j(w),
b_j(w), c_j(w), c_{j-1}(w)<\pi$,  the triangle $(\Delta
v_0v_jv_{j+1}, w*l)$ (or ($\Delta v_0v_jv_{j-1}, w*l)$)  is either
non-degenerate  or is degenerate  with $\pi$-angle at $v_j$, i.e.,
$b_{j+1}(w)=\pi$ (or $a_{j-1}(w)=\pi$ respectively). Therefore by
Proposition \ref{variation}(c), for small $t>0$, $\hat{w}*l$ is
still a generalized PL metric.

To see part (ii), we check the sum of opposite angles at
the following edges: $v_0v_{j-1}, v_0v_{j+1}$ and $v_0v_j$. At the
edge $v_0v_j$, due to (\ref{1235}) and continuity, we see
$a_j(\hat{w})+b_j(\hat{w}) < \pi$ for small $t>0$. At the edge
$v_0v_{j-1}$  (or similarly $v_0v_{j+1}$), by
 Proposition  \ref{variation}(c)
that $a_{j-1}(\hat{w}*l)$ and  $b_{j+1}(\hat{w}*l)$ are strictly
decreasing functions in $t>0$ and $b_{j-1}(\hat{w})=b_{j-1}(w)$,
we have
$$ a_{j-1}(\hat{w})+ b_{j-1}(\hat{w}) < a_{j-1}(w)+b_{j-1}(w) \leq
\pi.$$  Similarly, we have the Delaunay condition for $\hat{w}*l$
at the edge $v_0v_{j+1}$.

Finally, to see (iii), by Proposition  \ref{variation} and
(\ref{1235}), we have
$$\frac{d}{dt}|_{t=0}\alpha(\hat{w})=
\frac{d}{dt}|_{t=0}(c_j(\hat{w}))+\frac{d}{dt}|_{t=0}(c_{j-1}(\hat{w}))$$
$$=\cot(b_j(\hat{w}))+\cot(a_j(\hat{w})) >0.$$ Therefore, for small $t>0$,
$\alpha(\hat{w}) > \alpha(w)$.
\end{proof}




\begin{lemma}\label{ratio}
Let $(P_N, \T)$ be a star triangulation of an $N$-gon with
boundary vertices $v_1, ..., v_N$ labelled cyclically and one
interior vertex $v_0$ and $l: E(\T) \to \R_{>0}$ be a flat
generalized PL metric on $\T$. There is a constant $\lambda (l)$
depending on $l$ such that if $(P_N, \T, w*l)$ with $w:\{v_0, ...,
v_N\} \to \R$ is a generalized PL metric with zero curvature at
$v_0$,
then the ratio of edge lengths  satisfies
\begin{equation}\label{22} \frac{w*l(v_iv_{0})}{w*l(v_i v_{i+1})} \leq  \lambda(l)  \end{equation}
for all indices.
\end{lemma}

\begin{figure}[ht!]
\begin{center}
\begin{tabular}{c}
\includegraphics[width=0.3\textwidth]{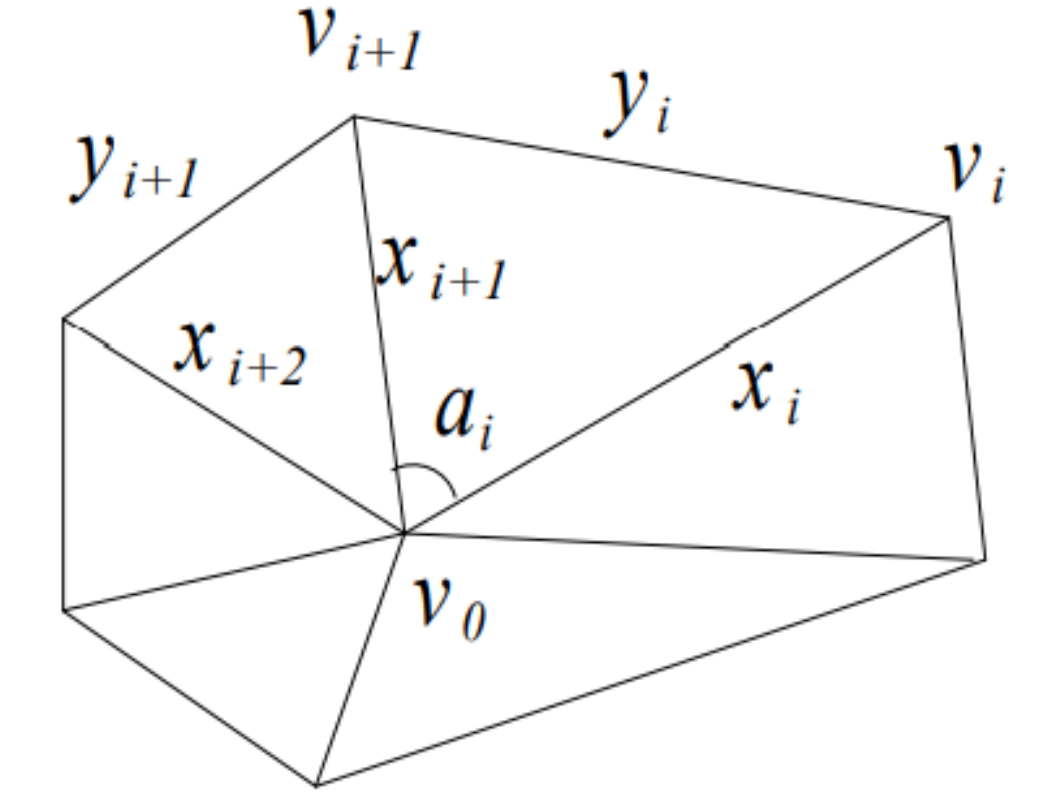}
\end{tabular}
\end{center}
\caption{Triangulated hexagon and length ratio} \label{3333}
\end{figure}

\begin{proof} Let $x_i(w)=w*l(v_0v_i)$ and $y_i(w)=w*l(v_iv_{i+1})$ be
the edge lengths
 in the metric $w*l$ where $v_{N+1}=v_1$.
By definition,
\begin{equation}\label{11}
 \frac{x_{i+2}}{y_{i+1}}=\lambda_i\frac{x_i}{y_{i}},
\end{equation} where $\lambda_i>0$ depends on $l$.
Then
\begin{equation}\label{601}
\frac{x_{i+1}}{y_{i+1}}\geq\frac{x_{i+2}-y_{i+1}}{y_{i+1}}=\frac{x_{i+2}}{y_{i+1}}-1=\lambda_i\frac{x_i}{y_{i}}-1.
\end{equation}\
We prove by contradiction. If the result of lemma 3.3 is not true, than there exists a sequence of conformal factors $w^{(n)}$ such that
$$
 \frac{x_i(w^{(n)})}{y_i(w^{(n)})}\rightarrow\infty,
$$
for some $i$. Without loss of generality, assume $i=1$ and
then by (\ref{601}) inductively we have
$$
\frac{x_{2}(w^{(n)})}{y_{2}(w^{(n)})}\rightarrow\infty,\quad
\frac{x_{3}(w^{(n)})}{y_{3}(w^{(n)})}\rightarrow\infty,...,
\frac{x_{N}(w^{(n)})}{y_{N}(w^{(n)})}\rightarrow\infty.
$$
Then the angle $a_i(w^{(n)})$ at $v_0$ in the triangle
 $\Delta v_0v_iv_{i+1}$ (in $w^{(n)}*l$ metric) converge to $0$, for any $i$.
But that contradicts the fact that the curvature
$2\pi - \sum_{i=1}^N a_i(w^{(n)})$  at $v_0$ is zero.
\end{proof}

The next result concerns linear discrete conformal factor and
spiral hexagonal triangulations.  It is a counterpart of Doyle
spiral circle packing in the discrete conformal setting. Unlike
Doyle spiral circle packing, not all choices of linear functions
produce generalized PL metrics.

We begin by recalling the developing maps. If $(S, \T, l)$ is a
flat generalized PL metric on a simply connected surface $S$
(i.e., $K_v=0$ for all interior vertices $v$), then a \it
developing map \rm $\phi: (S, \T, l) \to \C$ for $(\T,l)$ is an
isometric immersion determined by $|\phi(v)-\phi(v')|=l(vv')$ for
$v \sim v'$. It is constructed as follows. Fix a generalized
triangle $t \in \T$ and isometrically embeds $t$ to $\C$. This
defines $\phi|_t$.  If $s$ is a generalized triangle sharing a
common edge $e$ with $t$, we can extend $\phi|_t$ to $\phi|_{t
\cup s}$ by isometrically embedding $s$ to $\phi(s) \subset \C$
sharing the edge $\phi(e)$ with $\phi(t)$ such that $\phi(s)$ and
$\phi(t)$ are on different sides of $\phi(e)$. Since the surface
is simply connected, by the monodromy theorem, we can keep
extending $\phi$ to all triangles in $\T$ and produce a well
defined isometric immersion.
As a convention, if $\tau$ is triangle in  $\T$ and $l$ is a
generalized PL metric on $\T$, we use $(\tau, l)$ to denote the
induced generalized PL metric on $\tau$.

Given a lattice $L$ in $\C$, there exists a Delaunay triangulation
$\T_{st}=\T_{st}(L)$ of $\C$ with vertex set $L$ such that the
$\T_{st}$ is invariant under the translation action of $L$. In
particular $\T_{st}$ descends to a 1-vertex triangulation of the
torus $\C/L$. Therefore, the degree of each vertex $v \in \T_{st}$
is six, i.e., this triangulation is topologically the same as the
standard hexagonal triangulation of $\C$. Let $l_{0}: E(\T_{st})
\to \R_{>0}$ be the edge length function of $(\C, \T_{st}(L),
d_{st})$ where $d_{st}$ is the standard flat metric on $\C$. Let
$\tau$ be a triangle in $\T_{st}$ with vertices $0, u_1, u_2$.
Then $L =u_1\Z+u_2\Z$ and $\{u_1, u_2\}$ is called a \it geometric
basis \rm of $L$. Note that two vertices $v, v' \in L$ are joint
by an edge $e \in \T_{st}$ if and only if $v-v' \in \{\pm u_1, \pm
u_2, \pm (u_1-u_2)\}$.

\begin{proposition}\label{linear} Suppose $(\C,\T_{st}, l_0)$ is a
 hexagonal Delaunay triangulation of the plane with vertex set a lattice $V=u_1\Z +u_2\Z
$ where $\{u_1, u_2\}$ is a geometric basis.  Let  $w: V \to \R$
be a non-constant linear function $w(n u_1 +m u_2)= n
\ln(\lambda)+ m\ln(\mu)$, $m,n, \in \Z$, such that $w*l_{0}$ is a
generalized Delaunay PL metric on $\T_{st}$. Then the following hold.

(a) The generalized PL metric $(\T_{st}, w*l_{0})$ is flat.

Let $\phi$ be  the developing map for the flat metric
$(\T_{st},w*l_{0})$.

 (b) If there exists a  non-degenerate  triangle in the
generalized PL metric $w*l_{0}$, then there are two distinct
non-degenerate triangles $\sigma_1$ and $\sigma_2$ in $(\T_{st},
w*l_{0})$ such that $\phi(int(\sigma_1)) \cap \phi(int(\sigma_2))
\neq \emptyset$.

(c) Suppose all triangles in $w*l_0$ are degenerate.   Then there
exists an automorphism $\psi$ of the triangulation $\T_{st}$ such
that $w (\psi (n u_1+m u_2 )) =  n \ln(\gamma_1(V))+m\ln(\gamma_2(V))$
where $\gamma_i(V)$ are two explicit numbers depending only on $V$.
\end{proposition}

We remark that parts (a) and (b) for the lattice $\Z+e^{2\pi
i/3}\Z$ were proved in \cite{gsw}.

\begin{proof}  Consider two automorphisms $A$ and $B$ of the topological triangulation
$ \T_{st}$ defined by $A(v)=v+u_1$  and $B(v)=v+ u_2$ for $v \in
V$. By definition, we have $AB=BA$ and $A,B$ generate the group
$<A,B> \cong \Z^2$ acting on $\T_{st}$. Any triangle in $\T_{st}$
is equivalent, under the action of  $<A,B>$,  to exactly one of
the two triangles $T_1$ or $T_2$ where the vertices of $T_1$ are
$0, u_1, u_2$ and the vertices of $T_2$ are $0, -u_1, -u_2$. In
the generalized PL metric $w*l_0$, the maps $A$ and $B$ satisfy
$w*l_0(A(e))=\lambda^2 w*l_0(e)$ and $w*l_0(B(e))=\mu^2 w*l_0(e)$
for each edge $e \in \T$. It follows that for any triangle $\tau
\in \T_{st}$, the generalized triangle $(A(\tau), w*l_0)$ (resp.
$(B(\tau), w*l_0)$) is the scalar multiplication of $(\tau,
w*l_0)$ by $\lambda^2$ (resp. by $\mu^2$).  Hence
 there are only two similarity types of triangles in
$(\C, \T_{st}, w*l_0)$.  For each $v \in V$, the six angles at $v$
are congruent to the six inner angles in $T_1$ and $T_2$ in
$w*l_0$ metric. Therefore, $(\T, w*l_0)$ is a flat metric. See
Figure \ref{f666}(b).


\begin{figure}[ht!]
\begin{center}
\begin{tabular}{c}
\includegraphics[width=0.9\textwidth]{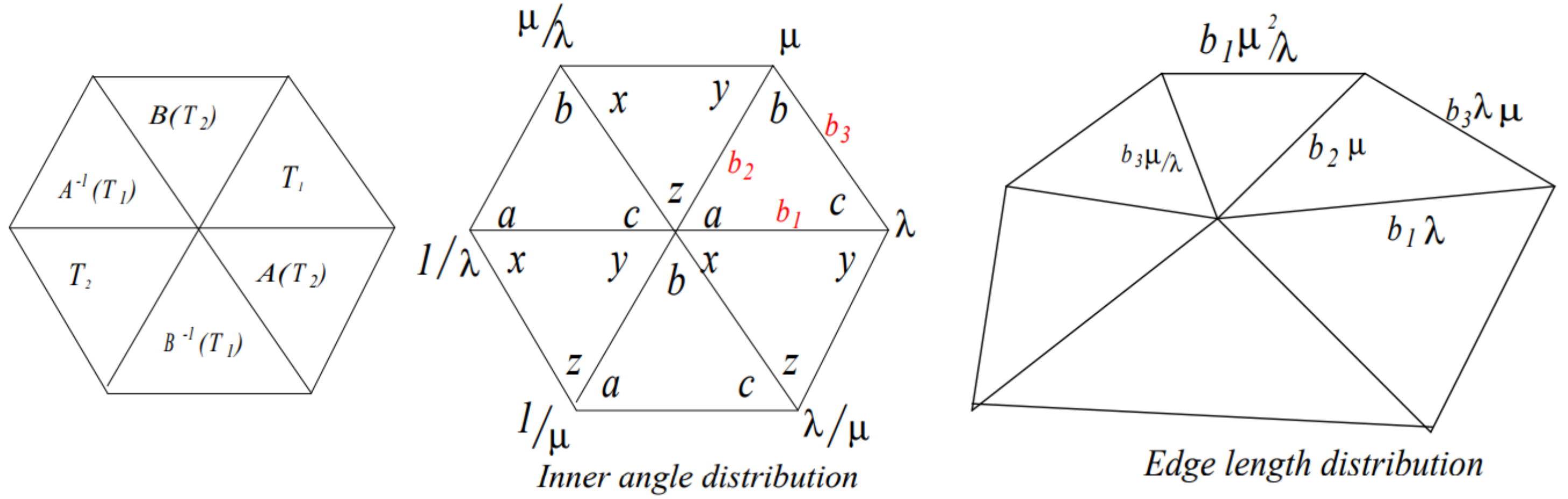}
\end{tabular}
\end{center}
\caption{Flatness of spiral hexagonal triangulations} \label{f666}
\end{figure}

 By the assumption that
$w$ is not a constant, we have $(\lambda, \mu) \neq (1,1)$. Say
$\lambda \neq 1$.
Using the developing map $\phi$, there exist two complex affine
maps $\alpha$ and $\beta$ of the complex plane $\C$ such that
$\phi A =\alpha \phi$ and $\phi B =\beta \phi$. Since $A$ is a
scaling by the factor $\lambda^2 \neq 1$ and $\phi$ is a local
isometry, the affine map $\alpha$ is of the form $\alpha(z)
=\lambda^*  z+ c$ where $|\lambda^*|=\lambda^2\neq 1$ and $\alpha$
has a unique fixed point $p\in \C$.  
By $AB=BA$, it follows $\alpha \beta =\beta \alpha$. Therefore,
from $\beta(p)=\beta \alpha(p)=\alpha\beta(p)$, we conclude
$\beta(p)=p$. After replacing the developing map $\phi$ by $\rho
\circ \phi$ for an isometry $\rho$ of $\C$, we may assume that
$\alpha$ and $\beta$ both fix
 $0$, i.e., $\alpha(z) =\lambda^* z$ and $\beta(z)=\mu^* z$ are
both scalar multiplications.  Let $G=<\alpha, \beta>$ be the
abelian group generated by $\alpha, \beta$ which acts on $\C$ by
scalar multiplications.

To see part (b), let $\Omega$ be the image $\phi(\C)$ of the
developing map which is invariant under the action of $G$. By the
assumption that there are non-degenerate triangles in
$(\T_{st},w*l_0)$,  the image $\Omega$ has non-empty interior.
There are two cases we have to consider. In the first case, there
exists a pair of integers $(n,m) \neq (0,0)$ so that $\alpha^n
\beta^m$ is the identity element in the group $G$. In this case,
we take $\sigma_1$ to be any non-degenerate triangle and
$\sigma_2=A^nB^m(\sigma_1)$. By definition, we have
$\phi(\sigma_1)=\phi(\sigma_2)$. Therefore, the result holds. In
the second case that  for all $(n, m) \neq (0,0)$, $\alpha^n
\beta^m \neq id$, i.e., the group $G$ is isomorphic to $\Z^2$.
Since both $\alpha(z)$ and $\beta(z)$ are scalar multiplications,
this implies that the action of the group $G$ on $int(\Omega)$ is
not discontinuous. In particular, for any non-empty open set $U
\subset \Omega$, there is $\alpha^n\beta^m \in G-\{id\}$ so that
$\alpha^n \beta^m (U) \cap U \neq \emptyset$. Take $\sigma_1$ to
be a  non-degenerate  triangle, $U=\phi(int(\sigma_1))$ and
$\sigma_2 =A^nB^m(\sigma_1)$. Then we have $\phi(int(\sigma_1))
\cap \phi(int(\sigma_2)) \neq \emptyset$.

To see part (c), since each triangle is degenerate, the inner
angles $a,b,c$ and $x,y,z$ of two triangles $T_1$ and $T_2$ are
$0$ or $\pi$ as shown in Figure \ref{f666}(b). Composing with  an
automorphism
 of $\T_{st}$, we may assume that $a=\pi$, and then by the Delaunay condition, $y\neq\pi$.

 There are two cases depending on
 $(x,y,z)=(\pi,0,0)$ or $(0,0,\pi)$. The two cases differ by the automorphism
$\rho$ of the lattice $u_1\Z+ u_2\Z$ and of $\T_{st}$ such that
$\rho(u_1)=u_2$, $\rho(u_2)=u_2-u_1$ and $\rho(0)=0$. Thus it
suffices to consider the  case: $z=\pi$.
Let the lengths of $u_1$, $u_2$ and $u_2-u_1$ in $l_0$-metric be
$b_1$, $b_2$ and $b_3$ respectively. The lengths of the corresponding edges in $w*l_0$ metric are $\lambda b_1$, $\mu b_2$ and $\lambda \mu b_3$.
 By the same computation, one works out the edge lengths of the triangle with vertices $0$,$u_2$ and $u_2-u_1$ in $w*l_0$ metric to be $\frac{\mu^2}{\lambda}b_1$, $\mu b_2$ and $\frac{\mu}{\lambda}b_3$. See Figure \ref{f666}(c).

We obtain two equations for edge lengths of
degenerate triangles: $\lambda b_1+\mu b_2=\lambda \mu b_3$ (due
to $a=\pi)$ and $ \frac{\mu^2}{\lambda} b_1 = \mu b_2+
\frac{\mu}{\lambda}b_3$ (due to $z=\pi)$. See Figure \ref{f666}(c).
These are same as $\lambda b_1+\mu b_2=\lambda \mu b_3$ and $\mu
b_1=\lambda b_2+b_3$. Solving $\mu$ in terms of $\lambda$, we
obtain a quadratic equation in $\lambda$:
\begin{equation}\label{solu}  b_2b_3 \lambda^2
+(b_3^2-b_1^2-b_2^2)\lambda -b_2b_3=0.\end{equation} Since
$b_i>0$, this equation has a unique positive solution which we
call $\gamma_1(V)$. The solution in $\mu$ is called
$\gamma_2(V)$.
 \end{proof}

\begin{figure}[ht!]
\begin{center}
\begin{tabular}{c}
\includegraphics[width=0.4\textwidth]{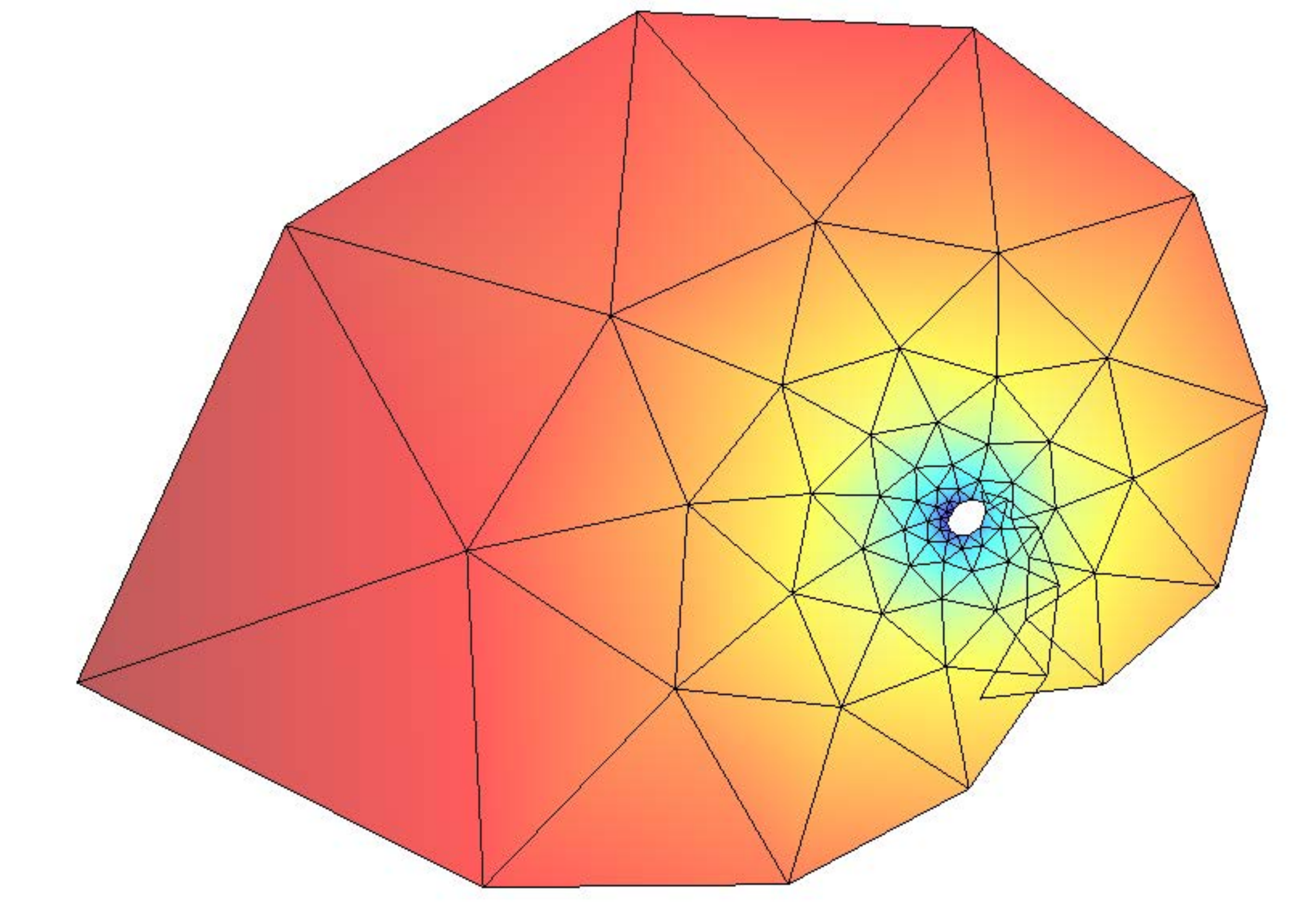}
\end{tabular}
\end{center}
\caption{Spiral hexagonal triangulations} \label{p}
\end{figure}

\section{Rigidity of hexagonal triangulations of the plane}

We begin with,
\begin{definition}
A flat generalized PL metric on a simply connected surface $(X,
\T, l)$ with developing map $\phi$ is said to be \it embeddable
\rm into $\C$ if for every simply connected finite subcomplex $P$
of $\T$, there exists a sequence of flat PL metrics on $P$ whose
developing maps $\phi_n$ converge uniformly to $\phi|_P$ and
$\phi_n: P \to \C$ is an embedding.
\end{definition}

For instance, all geometric triangulations of open sets in $\C$ are embeddable.  However, the spiral flat triangulations produced in Proposition 3.4 are not embeddable. The main result in this section works for embeddable flat PL metrics only.

The following lemma is a consequence of definition.

\begin{lemma}\label{embdev} Suppose $(X,\T,l)$ is a flat generalized PL metric on
a simply connected surface with a developing map $\phi$.

(a) Suppose $\phi$ is embeddable. If $t_1$, $t_2$ are two distinct
non-degenerate triangles or two distinct edges in $\T$, then
$\phi(int(t_1))\cap \phi(int(t_2)) =\emptyset$.

(b) If $\phi$ is  the pointwise convergent limit $\lim_{n\to
\infty}\psi_n$ of the developing maps $\psi_n$ of  embeddable flat
generalized PL metrics $(X, \T, l_n)$, then $(X,\T, l)$ is
embeddable.
\end{lemma}

\begin{proof} To see (a), if otherwise that
$\phi(int(t_1)) \cap \phi(int(t_2)) \neq \emptyset$, then $\phi$
is not embeddable. Indeed, take $P$ to be a finite simply
connected subcomplex containing $t_1$ and $t_2$, then the
developing maps $\phi_n$ defined on $P$ which converge uniformly
to $\phi|_P$ must satisfy $\phi_n(int(t_1)) \cap \phi_n(int(t_2))
\neq \emptyset$ for $n$ large. This contradicts that $\phi_n$ are
embedding.

Part (b) follows from the fact that $\psi_n$ converges to $\phi$
uniformly on compact subsets and the fact that if $\lim_{n \to
\infty}a_n=a$
 and $\lim_{m\to \infty} b_{n,m}=a_n$, then $a=\lim_{j\to \infty} b_{j, n_j}$ for some
 subsequence.

\end{proof}

Let $\T_{st}$ be a  hexagonal Delaunay triangulation of the plane
$S=\C$ with vertex set the lattice $V=\{ u_1n+ u_2m | n, m \in
\Z\}$ and  $l_0: E(\T_{st}) \to \R_{>0}$ be the edge length
function associated to $(S, \T_{st}, d_{st})$.  Given a flat
generalized PL metric $(S, \T_{st}, l)$, its normalized developing
map $\phi=\phi_l: S \to \C$ is a developing map such that
$\phi(0)=0$ and $\phi(u_1)$ is in the positive x-axis. Suppose
$\{u_1, u_2\}$ is a geometric basis of the lattice $u_1\Z+u_2\Z$.
Two vertices $v,v'$ are adjacent in $\T_{st}$, i.e., $v \sim v'$,
if and only if $v=v'+\delta$ for some $\delta \in \{\pm u_1, \pm
u_2, \pm(u_1-u_2)\}$. Given two vertices $v,v' \in V$, the
combinatorial distance $d_c(v,v')$ between $v,v'$ is the length of
the shortest edge path joining them.

 The goal of this section is to
prove the following stronger version of theorem \ref{rigid}.

\begin{theorem}\label{rigidity}  Suppose $(S, \T_{st}, l_0)$ is a
hexagonal Delaunay triangulation whose vertex set is a lattice in
$\C$  and $(S, \T_{st}, w*l_0)$ is a flat generalized Delaunay PL
metric on $\T_{st}$. If $(S, \T_{st}, w*l_0)$ is embeddable into
$\C$, then $w$ is a constant function.
\end{theorem}

We will deduce Theorem \ref{rigid} from Theorem \ref{rigidity} in
\S7.   Theorem \ref{rigidity} will be proved using several lemmas.

\subsection{Limits of discrete conformal factors}

The following lemma is a corollary of Theorem \ref{max}.

\begin{lemma} \label{6.2} Suppose $(S, \T_{st}, w*l_0)$ is
 a flat generalized Delaunay PL
metric surface.  Then for any $\delta \in V$ and $u: V \to \R$
defined by $u(v)=w(v+\delta) -w(v)$,  $u*(w*l_0)=(u+w)*l_0$ is  a flat
generalized Delaunay PL metric on $\T_{st}$. In particular, if
$u(v_0)=\max\{u(v) | v \in V\}$, then $u$ is constant.

\end{lemma}

The next lemma shows how to produce discrete conformal factors $w$
such that $w(v+\delta)-w(v)$ are constants.

\begin{lemma} \label{6.3} Suppose $w*l_0$ is a flat generalized Delaunay PL metric on
$\T_{st}$. Then for any $\delta \in \{\pm u_1, \pm u_2, \pm
(u_1-u_2)\}$, there exist $v_n \in V$ such that $w_n \in \R^V$ defined
by $w_n(v)=w(v+v_n)-w(v_n)$ satisfies

\noindent (a) for all $v \in V$, the following limit exists

$$ w_{\infty}(v) =\lim_{n \to \infty} w_n(v) \in \R,$$

\noindent (b) $w_n*l_0$ and
 $w_{\infty}*l_0$ are flat generalized Delaunay PL metric on $\T_{st}$,

\noindent (c) $w_{\infty}(v+\delta)-w_{\infty}(v)=a$  for all $v
\in V$ where $a =\sup\{ w(v+\delta)-w(v) | v\in V\}$,

\noindent (d)  the normalized developing maps $\phi_{w_{n}*l_0}$
of $w_{n}*l_0$ converges uniformly on compact sets in $S$ to the
normalized developing map $\phi_{\infty}$ of $w_{\infty}*l_0$. In
particular, if $(S, \T_{st}, w*l_0)$ is embeddable, then $(S,
\T_{st}, w_{\infty}*l_0)$ is embeddable.

\end{lemma} \begin{proof}  By Lemma \ref{ratio}, there is a constant $M=M(V)$ depending only on
the lattice $V=u_1 \Z + u_2\Z$ such that $a =\sup\{w(v+\delta)-w(v) | v\in V\} \leq M(V)$. Take $v_n \in V$ so that $$ w(v_n+\delta) -w(v_n) \geq a-\frac{1}{n}.$$
By definition, \begin{equation}\label{231} w_n(0)=0,  \quad  w_n(\delta) \geq a-\frac{1}{n},  \quad w_n(v+\delta)-w_n(v) \leq a,  \end{equation} and \begin{equation}   \sup\{|w_n(v) -w_n(v') | | v \sim v'\}  <  \infty. \end{equation}  By  Lemma \ref{ratio}, if $v \in V$ is of combinatorial distance $m$ to $0$, then, using $w_n(0)=0$, we have  \begin{equation}\label{cantor} |w_n(v)| \leq m M(V).\end{equation} By (\ref{cantor}) and the diagonal argument, we see that there exists a subsequence of $\{v_n\}$, still denoted by $\{v_n\}$ for simplicity, so that $w_n$ converges to
$w_{\infty} \in \R^V$ in the pointwise convergent topology. By
lemma \ref{6.2}, each $w_n*l_0$ is a flat generalized Delaunay PL
metric. By $\lim_{n \to \infty} w_n =w_{\infty}$ and continuity,
we conclude that $w_{\infty}*l_0$ is again a flat generalized
Delaunay PL metric on $\T_{st}$. By (\ref{231}),
$$ w_{\infty}(\delta)-w_{\infty}(0) =\max\{w_{\infty}(v+\delta)-w_{\infty}(v)| v \in V\}.$$
By Lemma \ref{6.2}, we see that conclusion (c) holds. Since the developing map $\phi_{w*l_0}$ depends continuously on $w
\in \R^V$,   $\lim_{n \to \infty} \phi_{w_n*l_0}(v)
=\phi_{\infty}(v)$  for each vertex $v \in V$.  On the other hand,
a developing map $\phi$ is determined by its restriction to $V$.
We see that $\phi_{w_n*l_0}$ converges to $\phi_{\infty}$
uniformly on compact subsets of the plane.  The last statement
follows from Lemma \ref{embdev}(b) since each $\phi_{w_n*l_0}$ is
embeddable by definition. \end{proof}

\subsection{Proof of Theorem \ref{rigidity}}
Suppose $w*l_0$ is a flat generalized Delaunay PL metric on
$\T_{st}$ with an embeddable developing map $\phi$.  Our goal is
to show that $w: V \to \R$ is a constant. Suppose otherwise, we
will derive a contradiction by showing that the developing map
$\phi$ is not embeddable.

Since $w$ is not a constant,  we can choose $\delta_1 \in \{ \pm
u_1, \pm u_2, \pm(u_1-u_2) \}$ such that $a_1=\sup\{
w(v+\delta_1)-w(v)|v \in V\}
> 0$. By Lemma \ref{6.3} applied to $w*l_0$ and $\delta=\delta_1$,
we produce a function $w_{\infty}: V \to \R$ so that
$w_{\infty}*l_0$ is a flat generalized Delaunay PL metric on
$\T_{st}$ and $w_{\infty}(v+\delta_1)=w_{\infty}(v)+a_1$ for all
$v\in V$. Now applying Lemma \ref{6.3} to $w_{\infty}*l_0$ with
$\delta_2 \in \{ \pm u_1, \pm u_2, \pm (u_1-u_2) \}-\{\pm
\delta_1\}$, we obtain a second function
$\hat{w}=(w_{\infty})_{\infty}: V \to \R$ and $b_1 \in \R$ such
that $\hat{w}*l_0$ is a flat generalized Delaunay PL metric on
$\T_{st}$ and
$$ \hat{w}(v+\delta_1)=\hat{w}(v)+a_1, \quad  \hat{w}(v+\delta_2)=\hat{w}(v)+b_1,$$
for all $v\in V$.  This shows that $\hat{w}: V \to \R$ is a
non-constant affine function, i.e., $\hat{w}(n+m e^{\pi i/3}) =
a_2n +b_2m +c_2$ for some $a_2, b_2, c_2 \in \R$.

Let $\hat{\phi}$, $\phi_{\infty}$ and $\phi$ be the normalized
developing maps for $\hat{w}*l_0$, $w_{\infty}*l_0$ and $w*l_0$
respectively.  Since $\phi$ is embeddable, by Lemmas \ref{6.3},
$\hat{\phi}$ and $\phi_{\infty}$ are embeddable.

If $\hat{w}*l_0$ contains a non-degenerate triangle, then by
Proposition   \ref{linear}, there exist two non-degenerate
triangles $t_1$ and $t_2$ in ($\T_{st}$, $\hat{w}*l_0)$ so that
$\hat{\phi}(int(t_1))\cap \hat{\phi}(int(t_2)) \neq \emptyset$. By
Lemma \ref{embdev}(a), this contradicts that $\hat{w}*l_0$ is
embeddable.



\begin{figure}[ht!] \begin{center}
\begin{tabular}{c}
\includegraphics[width=0.8\textwidth]{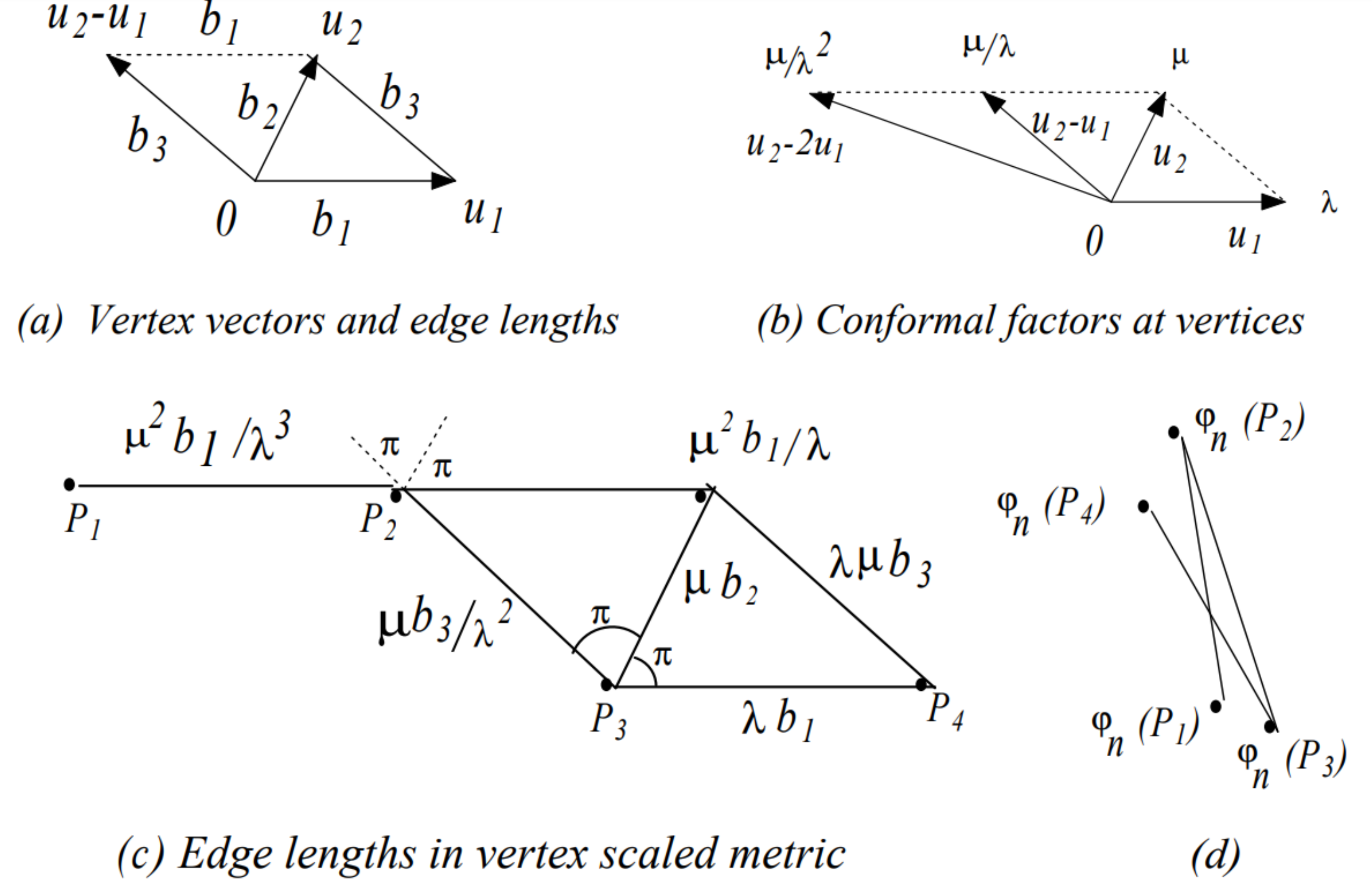}
\end{tabular} \end{center} \caption{Angles $a$ and $z$ are zero in $\hat{w}*l_0 $. Part (b)
is the developing image of corresponding set in $w*l_0$}
\label{f67}
\end{figure}
 Therefore all triangles in the generalized PL metric $\hat{w}*l_0$ are degenerate, i.e., all angles in triangles are either $0$ or
$\pi$. We will use the same notations used in the proof of Proposition \ref{linear}. By Proposition \ref{linear}(c) and Figure \ref{f67}, we may assume, after composing with an automorphism of $\T_{st}$ and subtracting by a constant, that $\hat{w}(  n u_1 +m u_2 )=n \ln(\gamma_1(V)) +m  \ln(\gamma_2(V))$ where $(\gamma_1(V), \gamma_2(V))$ are given by the
solutions  of (\ref{solu})  and the angles $a,b,c,x,y,z$ of $T_1$
and $T_2$ are $(a,b,c, x,y,z)=(\pi, 0, 0, 0, 0, \pi)$.

Let $P_1=u_2-2u_1$, $P_2=u_2-u_1$, $P_3=0$ and $P_4=u_1$ in $V$.
See Figure 6(c). In the case of $a=z=\pi$, we claim that the
length $\frac{\mu}{\lambda}b_3$ of the edge $P_2P_3$ is strictly
less than the sum of the lengths $\lambda b_1$ of the edge
$P_3P_4$ and $\frac{\mu^2}{\lambda^3}b_1$ of the edge $P_1P_2$,
i.e.,
\begin{equation} \label{912}\frac{\mu}{\lambda}b_3 <\lambda
b_1+\frac{\mu^2}{\lambda^3}b_1. \end{equation}

Indeed, by the equations $\lambda b_1+\mu b_2=\lambda\mu b_3$ and $\mu b_1=\lambda b_2+b_3$ derived in the proof of Proposition \ref{linear}, we obtain $$
\frac{b_3}{b_1}=\frac{ \lambda^2+\mu^2}{(1+\lambda^2)\mu}.$$
Equation (\ref{912}) says $$ \frac{b_3}{b_1} <
\frac{\lambda^4+\mu^2}{\lambda^2 \mu}.$$  Thus it suffices to show
that $\frac{ \lambda^2+\mu^2}{(1+\lambda^2)\mu}
<\frac{\lambda^4+\mu^2}{\lambda^2 \mu}$.  This is the same as,
$\lambda^2(\lambda^2+\mu^2)< (1+\lambda^2)(\lambda^4+\mu^4)$,
i.e., $\lambda^4+\lambda^2\mu^2 <
\lambda^4+\lambda^2\mu^2+\lambda^6+\mu^2$.  The last inequality
clearly holds since both $\lambda$ and $\mu$ are positive.

Now consider the oriented edge path $P_1P_2P_3P_4$ (oriented from
$P_2$ to $P_4$) in $\T_{st}$ and its image under the developing
map $\hat{\phi}$ of $\hat{w}*l_0$ in $\C$. By the assumption that
$a=z=\pi$, the angles of the polygonal path
$\hat{\phi}(P_1P_2P_3P_4)$ at $\hat{\phi}(P_2)$ and
$\hat{\phi}(P_3)$ are $2\pi$. See Figure 6(c). Also the sum of the lengths of
$\hat{\phi}(P_1P_2)$ and $\hat{\phi}(P_2P_4)$ is larger than the
length of $\hat{\phi}(P_2P_3)$ by the claim above.  On the other
hand, since $\hat{\phi}$ is embeddable, there exists a sequence of
flat PL metrics on $\T_{st}$ whose developing maps $\phi_n$ are
embedding and $\phi_n$ converges uniformly on compact sets to
$\hat{\phi}$. This implies, that for $n$ large the two line
segments $\phi_n(P_1P_2)$ and $\phi_n(P_3P_4)$ intersect in their
interiors. This contradicts the assumption that $\phi_n$ is an
embedding.

This ends the proof of Theorem \ref{rigidity}
\bigskip

\begin{remark} The above argument also gives a new proof of
Rodin-Sullivan's hexagonal circle packing theorem. \end{remark}

The following will be used to show that the limit of discrete
uniformization maps is conformal. Let $B_n(v)=\{ i \in V(\T_{st})
| d_c(i,v) \leq n\}$ and $\mathcal B_n(v)$ be the subcomplex of
$\T_{st}$ whose simplices have vertices in $B_n(v)$.

\begin{lemma}\label{hex} Take the standard hexagonal lattice $V=\Z +e^{2\pi/3}\Z$
and its associated standard hexagonal triangulation whose edge
length function is $l: V \to \{1\}$. There is a sequence $s_n$ of
positive numbers decreasing to zero with the following property.
For any integer $n$ and a vertex $v$, there exists $N=N(n, v)$ such that if $m \geq N$ and  $(\mathcal B_m(v), w*l_0)$ is a flat Delaunay triangulated PL
surface with embeddable developing map, then the ratio of the
lengths of any two edges sharing a vertex in $\mathcal B_m(v)$  is at most $1+s_n$.
\end{lemma}

The proof of the lemma is exactly the same as that of
Rodin-Sullivan \cite[pages 353-354]{RS}  since we have Lemma
\ref{ratio} and Theorem \ref{rigidity} which play the roles of
Rodin-Sullivan's Ring Lemma and rigidity of hexagonal circle
packing  in \cite[pages 352-353]{RS}).



\section{Existence of discrete uniformization metrics on polyhedral disks with special equilateral triangulations}

By a \it polygonal disk \rm  we mean a flat PL surface $(\mathcal
P, V,d)$ which is isometrically embedded in the complex plane $\C$
and $\mathcal P$ is homeomorphic to the closed disk.  The goal of this section is to prove the existence of a discrete conformal metric by regular subdividing the given triangulations.



An \it equilateral triangulation \rm $\T$ of a polyhedral surface
is a geometric triangulation whose triangles are equilateral. The
edge length function of an equilaterally triangulated connected polyhedral surface will be denoted by the constant function
 $l_{st}: E(\T) \to \R$. Given an equilateral Euclidean triangle $\Delta \subset \C$ and
$n \in \Z_{\geq 1}$,
the \it $n$-th standard subdivision \rm of $\Delta$ is the equilateral triangulation of $\Delta$ by $n^2$ equilateral triangles.
See Figure \ref{7}. If $\T$ is an equilateral triangulation of a
polyhedral surface, its \it $n$-th standard subdivision\rm, denoted
by $\T_{(n)}$, is the equilateral triangulation obtained by
replacing each triangle in $\T$ by its $n$-th standard subdivision.
We use $V_{(n)}$ to denote
$V(\T_{(n)})$. 

\begin{figure}[ht!]
\begin{center}
\begin{tabular}{c}
\includegraphics[width=0.8\textwidth]{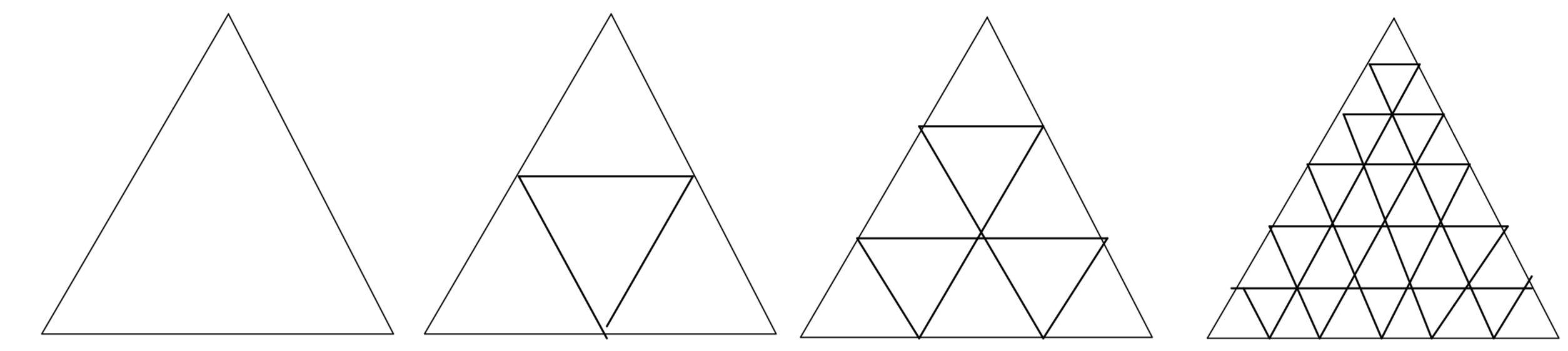}\\
\quad
1st\quad \quad\quad\quad\quad\quad
2nd\quad \quad\quad\quad\quad\quad
3rd\quad \quad\quad\quad\quad\quad$~~~~$
6th
\end{tabular}
\end{center}
\caption{The standard subdivisions} \label{7}
\end{figure}

The main result of this section is  the following theorem.

\begin{theorem}\label{1.1} Suppose $(\mathcal P, \T, l_{st})$ is a
flat polygonal disk with an equilateral triangulation $\T$
 such that exactly three boundary vertices $p,q,r$ have
curvature $\frac{2\pi}{3}$. 
Then for sufficiently large $n$, there is discrete conformal
factor $w_n: V_{(n)}\to \R$ for the $n$-th standard subdivision
$(\mathcal P, \T_{(n)}, l_{st})$  such that the discrete
 curvature $K$ of $w_n*l_{st}$ satisfies

 (a) $K_i=0$ for all $i \in V_{(n)}-\{p,q,r\}$,

 (b) $K_i=\frac{2\pi}{3}$ for all $i \in \{p,q,r\}$, and

 (c) there is a constant $\epsilon_0>$ independent of $n$ such that all inner angles of triangles in $(\T_{(n)}, w_n*l_{st})$ are in
 the interval $[\epsilon_0, \pi/2+\epsilon_0]$,  the sum of two angles
 facing each interior edge is at
 most $\pi-\epsilon_0$ and each angle facing a boundary edge is at most
 $\pi/2-\epsilon_0$.
\end{theorem}

Conditions (a) and (b) imply that the underlying metric space of
$(\mathcal P, \T_{(n)}, w_n*l_{st})$ is an equilateral triangle.
Condition (c) says that the metric doubles of $(\mathcal P,
\T_{(n)}, l_{st})$ and $(\mathcal P, \T_{(n)}, w_n*l_{st})$ are
two Delaunay triangulated polyhedral 2-spheres differing by a vertex scaling.

There are two steps involved in the construction of the discrete
conformal factor $w_n$ in Theorem \ref{1.1}.  In the first step,
we produce a discrete conformal factor $w^{(1)}: V_{(n)} \to \R$
such that $w^{(1)}$ vanishes outside the union of combinatorial
balls of radius $[n/3]$ (the integral part of $n/3$)  centered at non-flat vertices $v \neq
p,q,r$ and the discrete curvature $K_i(w^{(1)}*l_{st}) =0$ if
$d_c(i,v) < [n/3]$ and $K_i(w^{(1)}*l_{st}) = O(1/\sqrt{\ln(n)})$
if $d_c(i, v)=[n/3]$. This step diffuses the non-zero discrete
curvatures $\pi/3, -\pi/3$ and $-2\pi/3$ (at non-flat vertices
$v$) to small curvatures  at vertices defined by $d_c(i,v)=[n/3]$.
In the second step, by choosing $n$ large such that all curvatures
are very small, we use a perturbation argument to show that there
is $w^{(2)}: V_{(n)} \to \R$ such that $w^{(2)}*(w^{(1)}*l_{st})$
satisfies the conditions in Theorem \ref{1.1}.  The required
discrete conformal factor $w_n$ is $w^{(1)}+w^{(2)}$ since
$(w^{(2)} +w^{(1)})*l_{st}= w^{(2)}*(w^{(1)}*l_{st}).$

The basic tools to be used for proving Theorem \ref{1.1} are
discrete harmonic functions, their gradient estimates and ordinary
differential equations (ODE). We begin by recalling the related
material.

\subsection{Laplace operator on a finite graph}

Given a graph $(V, E)$,  the set of all oriented edges in $(V,E)$
is denoted  by $\bar{E}$.   If $i \sim
 j$ in $V$, we use $[ij] \in \bar{E}$ to denote the oriented edge from $i$ to $j$.
If $x \in \R^V$ and $y \in \R^{\bar{E}}$, we use $x_i$
 and $y_{ij}$ to denote $x(i)$ and $y([ij])$ respectively.
A \it conductance \rm on $G$ is a function $\eta: \bar{E} \to
\R_{\geq 0}$ so that $\eta_{ij}=\eta_{ji}$.

\begin{definition} Given a finite graph $(V,E)$ with a conductance $\eta$,
the   gradient  $\triangledown: \R^V \to \R^{\bar{E}}$ is the
linear map $$ (\triangledown f)_{ij} =\eta_{ij}(f_i-f_j),$$
 the  Laplace operator  associated to $\eta$ is the linear map
$\triangle: \R^V \to \R^V$ defined by
$$ (\triangle f)_i =\sum_{j \sim i}
\eta_{ij}(f_i-f_j),$$ and the  Dirichlet energy  of  $f \in \R^V$
on $(V,E,\eta)$ is $$ \mathcal E(f) =\frac{1}{2} \sum_{ i \sim j}
\eta_{ij}(f_i -f_j)^2.$$


\end{definition}

The following is well known (see \cite{chung}).

\begin{proposition}[Green's identity]\label{green} Given a finite
graph $(V,E)$ with a conductance $\eta$,

(a)  for any subset $V_0 \subset V$,
$$ \sum_{ i \in V_0}f_i (\triangle g)_i -g_i ( \triangle f)_i =\sum_{ i
\in V_0, j \sim i, j \notin V_0} \eta_{ij}(g_if_j-f_ig_j).$$

(b)  $\sum_{i \in V} (\triangle f)_i=0.$

\end{proposition}

Given a set $V_0 \subset V$ and $g: V_0 \to \R$, the \it Dirichlet
problem \rm asks for a function $f: V \to \R$ so that

\begin{equation}\label{dp} (\triangle f)_i=0, \text{$\forall i \in
V-V_0$, and }  f|_{V_0}=g. \end{equation}
 The \it Dirichlet principle \rm
states that solutions $f$ to the Dirichlet problem (\ref{dp}) are
the same as minimum points of the  Dirichlet energy function restricted to
the affine subspace $\{ h \in \R^V | h|_{V_0}=g\}$, i.e.,
\begin{equation}\label{dprinciple}
\mathcal E(f) = \min\{ \mathcal E(h)| h \in \R^V \quad \text{and}
\quad h|_{V_0}=g \}.
\end{equation}
In particular, the Dirichlet problem (\ref{dp}) is always
solvable.

 A subset $U \subset V$ in a graph $(V,E)$
is called \it connected \rm if any two vertices $i, j \in U$ can
be joint by an edge path whose vertices are in $U$. For instance,
a connected  graph $(V,E)$ means $V$ is a connected.
The following is well known (see \cite{chung}).

\begin{proposition}\label{mphar} Suppose $(V,E)$ is a finite
connected graph with a conductance $\eta_{ij}>0$ for all edges
$[ij]$ and $V_0 \subset V$. Let $f$ be a solution to  the
Dirichlet problem (\ref{dp}). Then,

(a) (Maximum principle) for $V_0\neq \emptyset$,
$$ \max_{i \in V}
f_i =\max_{ i \in V_0} f_i.$$

(b) (Strong maximum principle) If $V-V_0$ is connected and
$\max_{i\in V-V_0} f_i =\max_{i \in V_0} f_i$,
 then $f|_{V-V_0}$ is a constant function.
\end{proposition}



\subsection{A system of ODE associated
to discrete conformal change} Let $(S, \T, l)$ be a compact
connected polyhedral surface with discrete curvature $K^0$. Given
a subset $V_0 \subset V$ and a function $K^*: V-V_0 \to (-\infty,
2\pi)$, we try to find a function $w: V \to \R$ such that $w*l$ is
a PL metric whose curvature $K(w)$ is equal to $K^*$ on $V-V_0$
and $w|_{V_0}=0$. In the PL metric $w*l$, let
$\theta^i_{jk}=\theta^i_{jk}(w)$ be the angle at vertex $i$ in the
triangle $\Delta ijk$ and $\eta_{ij}=\eta_{ij}(w)$ be
$\cot(\theta^k_{ij})+\cot(\theta^l_{ij})$ if $[ij]$ is an interior
edge and $\eta_{ij}=\cot(\theta^k_{ij})$ if $[ij]$ is a boundary
edge. The associated Laplacian  $\Delta: \R^V \to \R^V$ is
$(\Delta f)_i =\sum_{j \sim i} \eta_{ij}(f_i-f_j)$.   We will
construct $w$ by choosing a smooth 1-parameter family $w(t)\in
\R^V$ such that $w(0)=0$ and $w(t)*l$ is a PL metric whose
curvature $K_i(t)=K_i(w(t)*l_{st})$ satisfies \begin{equation}
\label{ode1} \forall i\in V-V_0, \quad K_i(t)=(1-t)K_i^0+tK_i^*;
\quad \text{and} \quad \forall i \in V_0, \quad
 w_i(t)=0.
\end{equation}  The required vector  $w$ is defined to be $w(1)$.
 Note that by definition $K(0)=K^0$.
Due to the curvature evolution equation (\ref{curvatureevol}) that
$\frac{dK_i(t)}{dt}=\sum_{j \sim i} \eta_{ij}(w(t)) (w_i'-w_j')$
where $w'_i(t)=\frac{dw_i(t)}{dt}$, we obtain the following system
of ODE in $w(t)$ equivalent to (\ref{ode1}):
\begin{equation}\label{ode3}
\forall i \in V-V_0, \quad \sum_{j \sim i}
\eta_{ij}(w_i'-w_j')=K^*_i-K_i^0;   \quad \forall i\in V_0, \quad
w'_i(t)=0; \quad \text{and  $w(0)=0$}.
\end{equation}
Using  $\triangle f$, we can write Equation (\ref{ode3}) as
\begin{equation}\label{ode4}
\forall i \in V-V_0, \quad (\Delta w')_i=K^*_i-K_i^0;  \quad
\forall i\in V_0, \quad w'_i(t)=0; \quad \text{and  $w(0)=0$.}
\end{equation}

We will show, under some assumptions on $(\T, l)$, that the
solution to (\ref{ode3}) exists for all $t \in [0,1]$.

Let $W \subset \R^V$ be the open set
\begin{equation}\label{wdomain} W=\{ w \in \R^V | w*l  \text{ is a
PL metric on $\T$ and $\eta_{ij}(w)>0$ for all edges $[ij]$} \}.
\end{equation}

\begin{lemma} \label{odeexist} Suppose  $V_0 \neq
\emptyset$ and $0\in W$.  The initial valued problem (\ref{ode3})
defined on $W$ has a unique solution in a maximum interval $[0,
t_0)$ with $t_0>0$ such that if $t_0<\infty$, then either
$\liminf_{t \to t_0^-} \theta^i_{jk}(w(t)) = 0$ for some angle
$\theta^i_{jk}$ or $\liminf_{t \to t_0^-} \eta_{ij}(w(t))=0$ for
some edge $[ij]$.
\end{lemma}

\begin{proof}
Indeed Equation (\ref{ode3}) can be written as $Y(w)\cdot
w'(t)=\beta$ and $w(0)=0$  where $Y(w)$ is a square matrix valued
smooth function of $w \in W$ and $w'(t)$ is considered as a column
vector. We claim that $Y(w)$ is an invertible matrix for $w \in
W$. If  $Y(w)$ is invertible, then (\ref{ode3}) can be written as
$w'(t)=Y(w)^{-1} \beta$ and  by the Picard's
existence theorem, there exists an interval on which the ODE (\ref{ode3}) has a solution. Now $Y(w)$ is invertible if and only if  the
following system of linear equations has only trivial solution
$x=0$,
\begin{equation}\label{ode5} Y(w) \cdot x=0.
\end{equation}
By (\ref{ode3}), Equation (\ref{ode5}) is the same as $(\Delta
x)_i=0$ for  $i \in V-V_0$ and $x_i=0$ for  $i\in V_0$.
Furthermore $w \in W$ implies $\eta_{ij}(w)>0$ for all edges
$[ij]$. By the maximum principle (Proposition \ref{mphar}), we see
that $x=0$.

If $t_0<\infty$ and $t \uparrow t_0$, then $w(t)$ leaves every
compact set in $W$.    For each $\delta>0$, we claim that
$W_{\delta}=\{ w \in W|
 \theta^{i}_{jk} \geq \delta,  |w_{i}| \leq \frac{1}{\delta},\eta_{ij} \geq \delta
 \}$ is compact. Clearly $W_{\delta}$ is bounded by definition.
To see that $W_{\delta}$ is closed in $\R^V$, take a sequence $x_n
\in W_{\delta}$ such that $\lim_{n \to \infty} x_n =y \in \R^V$.
Then $y*l$ is a generalized PL metric with all angles
$\theta^i_{jk} \geq \delta$. Since each degenerate  triangle has
an angle which is zero, therefore $y*l$ is a PL metric. Also by
continuity, we have $\theta^i_{jk}(y)\geq \delta$,
$\eta_{ij}(y)\geq \delta$ and $|y_{i}|\leq \frac{1}{\delta}$,
i.e., $y \in W_{\delta}$. Since $w(t)$ leaves every $W_{\delta}$
for each $\delta >0$, one of the following three occurs:
$\liminf_{t \to t_0^-} \theta^i_{jk}(w(t))=0$ for some
$\theta^i_{jk}$, or $\liminf_{t \to t_0^-} \eta_{ij}(w(t))=0$ for
some edge $[ij]$, or $\limsup_{t \to t_0^-} |w_i(t)|=\infty$ for
some $i_0 \in V$. However $\limsup_{t \to t_0^-} |w_{i_0}(t)|=\infty$
for one vertex $i_0$ implies that $\liminf_{t \to t_0^-}
\theta^l_{jk}(w(t))=0$ for some $\theta^l_{jk}$. Indeed, if
otherwise, $\liminf_{t \to t_0^-} \theta^l_{jk}(w(t))\geq
\delta>0$ for all $\theta^l_{jk}$ for some $\delta$. It is well
known  that in a Euclidean triangle whose angles are at least $\delta$,
the ratio of two edge lengths is at most $\frac{1}{\sin(
\delta)}$. Therefore, in each triangle $\Delta v_iv_jv_k$  in
$\T$, we have $e^{w_i(t)} \leq
e^{w_j(t)}\frac{l(v_jv_k)}{l(v_iv_k)\sin(\delta)}$. Since
$w_j(t)=0$ for $j \in V_0$ and the surface $S$ is connected, we
conclude that all $w_k(t)$, $k\in V$, are bounded for all $t$.
This contradicts $\limsup_{t \to t_0^-} |w_{i_0}(t)|=\infty$.
 \end{proof}

\subsection{Standard subdivision of an equilateral triangle}


%

\begin{theorem}\label{1.2}  Let $S=\Delta ABC$ be an
equilateral triangle,  $\T$ be the $n$-th standard subdivision of
$S$ with the associated PL metric $l_{st}:V= V(\T) \to \{
\frac{1}{n}\}$ and $V_0=\{ v \in V| v$ is in the edge $BC$ of the
triangle $\Delta ABC$\}. Given any $\alpha \in [\frac{\pi}{6},
\frac{\pi}{2}]$, there exists a smooth family of vectors $w(t) \in
\R^V$ for $t\in [0,1]$ such that $w(0)=0$ and  $w(t)*l_{st}$ is a
PL metric on $\T$ with curvature $K(t)=K(w(t)*l_{st})$ satisfying,

(a) $K_A(t)= -t\alpha+(2+t)\frac{\pi}{3}$ (angle at $A$ is
$t\alpha+(1-t)\frac{\pi}{3}$),

(b)  $K_i(t)$=0  for all $i \in V-\{A\}\cup V_0$,

(c) $w_i(t)=0$ for all $i \in V_0$,

(d) all inner angles $ \theta^i_{jk}(t)$ in metric $w(t)*l_{st}$
are in the interval  $ [\frac{\pi}{3}-|\alpha-\frac{\pi}{3}|,
\frac{\pi}{3}+|\alpha-\frac{\pi}{3}|] \subset [\frac{\pi}{6},
\frac{\pi}{2}]$,

(e)  $\theta^i_{jk}(t) \leq \frac{59 \pi}{120}$ for $i \neq A$,

(f)  $|K_i(t) -K_i(0)| \leq \frac{ 2000}{\sqrt{ \ln(n)}}$ for $i
\neq A$ and
\begin{equation}\label{totalcur} \sum_{i \in V_0}
|K_i(t)-K_i(0)|\leq \frac{\pi}{6}.  \end{equation}
\end{theorem}

\begin{remark} The discrete conformal map from $(\Delta ABC, \T,
l_{st})$ to $(\Delta ABC, \T, w(1)*l_{st})$ is a discrete
counterpart of the analytic function $f(z)=z^{3\alpha/\pi}$.

\end{remark}

\begin{figure}[ht!]
\begin{center}
\begin{tabular}{c}
\includegraphics[width=1\textwidth]{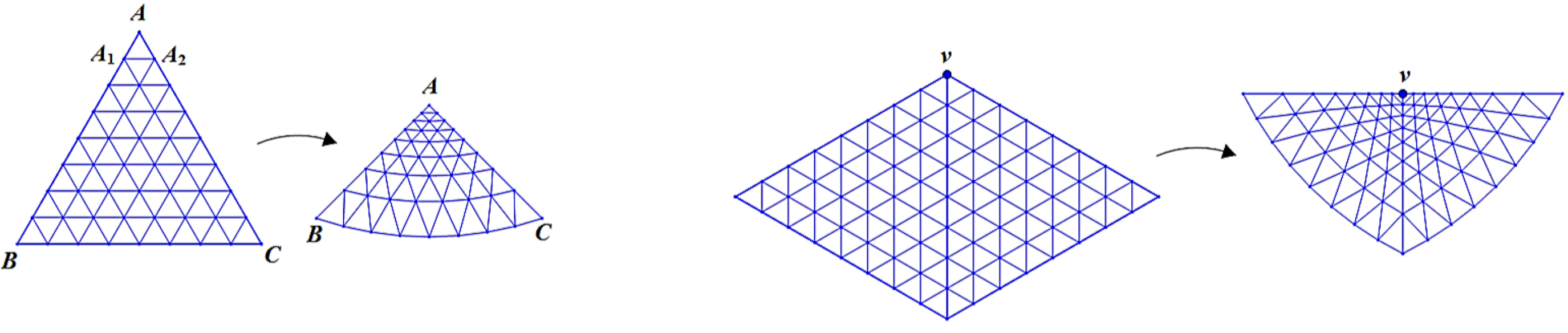}
\end{tabular}
\end{center}
\caption{Discrete conformal maps of equilateral triangles and
their unions} \label{p4}
\end{figure}


Our proof of Theorem \ref{1.2} relies on the following two lemmas about estimates on discrete harmonic functions on $\mathcal T$.

\begin{lemma} \label{1.3}
Assume $\Delta ABC,n,\mathcal T,V_0$ are as given in Theorem \ref{1.2}.
Let $\tau: \T \to \T$ be the involution
induced by the reflection of $\Delta ABC$ about the angle bisector
of $\angle BAC$ and $\eta: E \to \R_{\geq 0}$ be a conductance so
that $\eta \tau =\eta$ and $\eta_{ij}=\eta_{ji}$. Let $\Delta:
\R^V \to \R^V$ be the Laplace operator defined by $(\Delta f)_i
=\sum_{ j \sim i} \eta_{ij}(f_i -f_j)$. If $f \in \R^V$ satisfies
$(\Delta f)_i=0$ for $i \in V-\{A\}\cup V_0$ and $f|_{V_0}=0$,
then for all edges $[ij]$, the gradient $(\triangledown
f)_{ij}=\eta_{ij}(f_i-f_j)$ satisfies
\begin{equation}\label{90e} |\eta_{ij}(f_i -f_j)| \leq
\frac{1}{2}|\Delta(f)_A|.
\end{equation}
\end{lemma}

\begin{lemma}\label{891}
Assume $\Delta ABC,n,\mathcal T,V_0$ are as given in Theorem \ref{1.2}.
Let $\eta: E(\T) \to [\frac{1}{M}, M]$ be a conductance
function  for some $M>0$ and  $\Delta$ be the Laplace operator on
$\R^V$ associated to $\eta$. If $f:V \to \R$ solves the Dirichlet
problem $(\Delta f)_i=0, \forall i \in V-\{A\}\cup V_0$,
$f|_{V_0}=0$ and $(\Delta f )_A=1$, then for all $u \in V_0$,
$|(\Delta f)_u|\leq \frac{20M}{\sqrt{\ln n}}.$
\end{lemma}

We will prove Lemma \ref{1.3}, Lemma \ref{891}, and Theorem \ref{1.2} in order.

The simplest way to see Lemma \ref{1.3} is to use the theory of electric
network. We put a resistance of $\frac{1}{\eta_{ij}}$ Ohms at the
edge $[ij]$  (if $\eta_{ij}=0$, the resistance is $\infty$, or
remove edge $[ij]$ from the network). Now place a one-volt battery
at vertex $A$ and ground every vertex in $V_0$. Then Kirchhoff's
laws show that the voltage $f_i$ at the vertex $i$ solves the
Dirichlet problem $( \Delta f)_i=0$ for all $i \in V-\{A\}\cup
V_0$, $f_A=1$ and $f|_{V_0}=0$. Ohm's law says
$\eta_{ij}(f_i-f_j)$ is the electric current through the edge
$[ij]$. Since the resistance is symmetric with respect to the
symmetry $\tau$,  the currents in the network are the same as the
currents in the quotient network $\T/\tau$. In the quotient
network $\T/\tau$, there is only one edge $e_A$ from the vertex A.
Therefore, the current through any edge is at most the current
$\frac{1}{2}|(\Delta f)_A|$ through $e_A$ (in the network
$\T/\tau$). This shows $|\eta_{ij}(f_i-f_j)| \leq
\frac{1}{2}|(\Delta f)_A|$.

\begin{proof}[Proof of lemma \ref{1.3}] Removing all edges $[ij]$ for which $\eta_{ij}=0$ from
the graph $(V, E)$, we obtain a finite collection of disjoint
connected subgraphs $\Gamma_1, ..., \Gamma_N$ from $(V, E)$.  By
construction, the associated Laplace operators on $\Gamma_i$ with
conductance $\eta|_{E(\Gamma_i)}$ is the restriction of the
Laplace operator $\Delta$ to $V(\Gamma_i)$. By the maximum
principle (Proposition \ref{mphar}), the function
$f|_{V(\Gamma_m)}$ is a constant and (\ref{90e}) holds unless
$\Gamma_m$ contains the vertex $A$ and some vertex in $V_0$.
Therefore, it suffices to prove the lemma for those edges $[ij]$
in the connected graph $\Gamma_m=(V', E')$ such that $A \in V'$
and $V'\cap V_0 \neq \emptyset$.  Let $A_1, A_2=\tau(A_1)$ be the
vertices adjacent to $A$. Since $\tau(A)=A$,
 $\eta \tau=\eta$ and $V' \cap V_0 \neq \emptyset$, we have $\tau(\Gamma_m)=\Gamma_m$ and $A_1, A_2
\in V'$.

We will work on the graph $\Gamma_m=(V',E')$ from now on.  Using
the maximum principle for $f -f\tau$, we see that $f=f\tau$. By
replacing $f$ by $-f$ if necessary, we may assume that $f_A >0$.
By the maximum principle, we  have that  $0 \leq f_i < f_A$ for
all $i \in V'-\{A\}$.

Take an edge $[ij]$ in the graph $\Gamma_m$.  If $\tau
\{i,j\}=\{i,j\}$, then $\tau_i=j$ and $\tau_j=i$. This implies
$f_i=f \tau_i=f_j$ and (\ref{90e}) holds.  If $\tau \{i,j\}
=\{i',j'\} \neq \{i,j\}$, say $\tau_i=i', \tau_j=j'$, then
$f_i=f_{i'}, f_j =f_{j'}$. We may assume that $f_i \leq f_j$. If
$f_i=f_j$, then (\ref{90e}) holds. Hence we may assume $f_i < f_j$.
If $j=A$, then $i=A_1$ or $A_2$. Due to $f_{A_1}=f_{A_2}$, then
(\ref{90e}) holds. If $j \neq A$, then by the maximum principle
applied to $f$ on the subgraph $(V'-\{A\}, E'-\{AA_1, AA_2\})$, we
conclude that $f_{A_1} \geq f_j > f_i$. Let $U=\{ k \in V'-\{A\} |
f_k
> f_i\}$. By definition, $j,j', A_1, A_2 \in U$,  $i,i',A \notin
U$, and $V_0 \cap U=\emptyset$.  This shows $(\Delta
 f)_k=0$ for all $k \in U$ and hence $\sum_{k \in U}
( \Delta f)_k=0$.  By Green's formula (\ref{green}), $$\sum_{k \in
U} (\Delta f)_k=\sum_{k \in U, l \notin U, k\sim l}
\eta_{kl}(f_k-f_l) =0.$$ If  $l \notin U \cup\{A\}$, then by
definition $f_i \geq f_l$. Therefore, if $k \in U$, $k \sim l$,
and $l \notin U \cup\{A\}$, then $f_k>f_i\geq f_l$. This shows,
$$ 0=\sum_{k \in U, l \notin U, l \sim k} \eta_{kl}(f_k -f_l)$$
$$ =\sum_{k \in U, l \notin U \cup \{A\}, l \sim k} \eta_{kl}(f_k
-f_l)+\sum_{ k \sim A} \eta_{kA}(f_k -f_A)$$
$$\geq (\triangledown f)_{ji}+(\triangledown f)_{j'i'} -(\Delta
f)_A.$$  Therefore, $|(\Delta f)_A |\geq 2 |(\triangledown
f)_{ij}|$ since $(\triangledown f)_{ij}=(\triangledown f)_{i'j'}$.
\end{proof}

\begin{proof}[Proof of Lemma \ref{891}]
For the given $u \in V_0$, construct a function $g: V \to \R$ by
solving the Dirichlet problem: $(\Delta g)_i =0, \forall i \in
V-V_0$, $g_u=1$ and $g|_{V_0-\{u\}}=0$.  By the maximum principle
(Proposition \ref{mphar}),  $0 \leq g_i \leq 1$ for all $i$. Using
 Green's identity that $\sum_{i \in V} [f_i (\Delta g)_i
-g_i(\Delta f)_i]=0$, we obtain $g_A(\Delta f)_A +g_u(\Delta
f)_u=0$. Since $(\Delta f)_A=1$ and $g_u=1$, we see
$$ (\Delta f)_u =-g_A.$$

\begin{figure}[ht!]
\begin{center}
\begin{tabular}{c}
\includegraphics[width=0.37\textwidth]{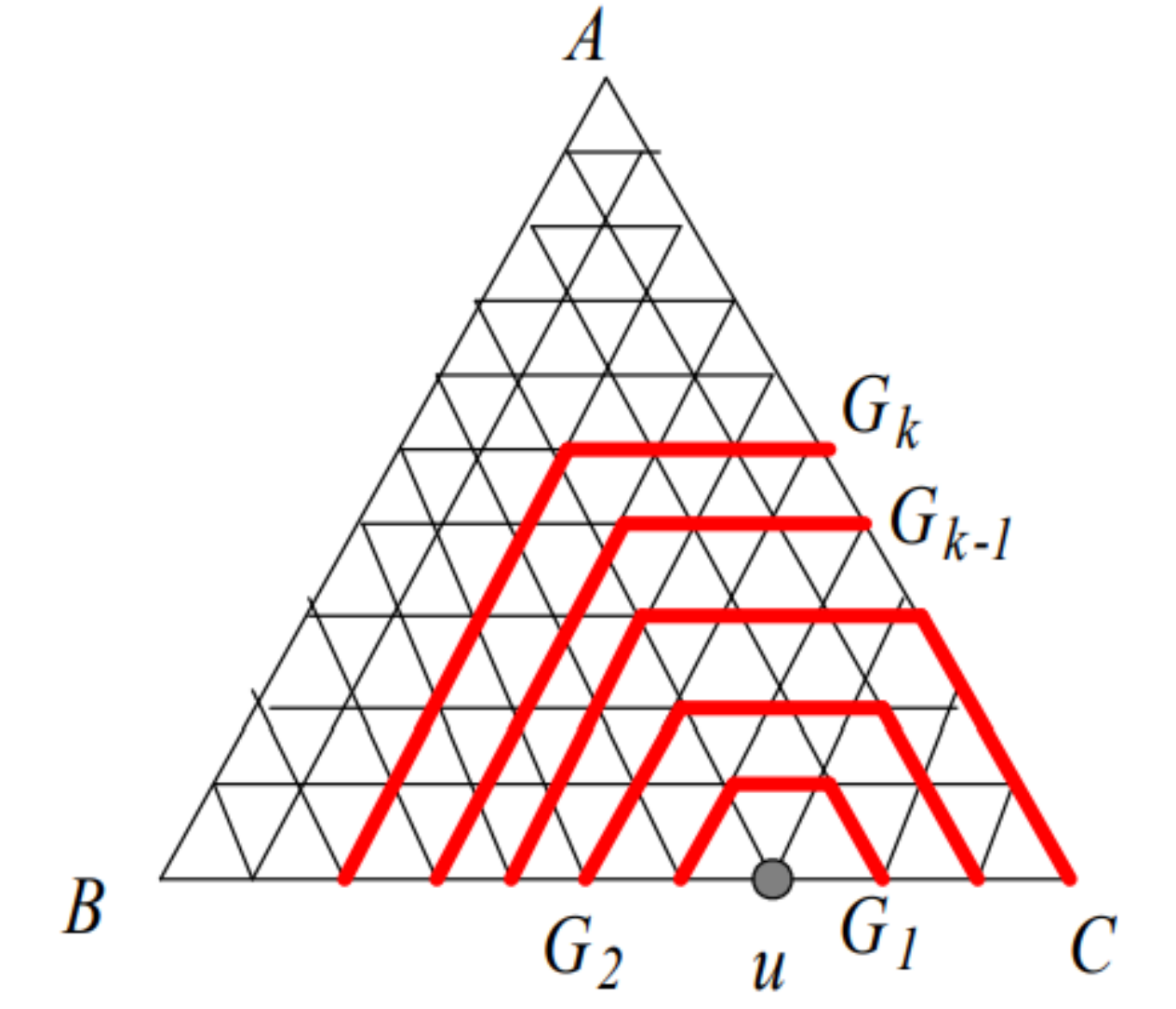}
\end{tabular}
\end{center}
\caption{Layers in triangle ABC} \label{subdivion}
\end{figure}

Therefore, it suffices to show that $|g_A| \leq
\frac{20M}{\sqrt{\ln n}}.$  For this purpose, take $k \leq
[\frac{n}{2}]$ and define $U_k =\{ i \in V| d_c(i,u)=k\}$ where
$d_c(i,j)$ is the combinatorial distance in the graph $\T^{(1)}$.
Let $G_k$ be the subgraph of $\T^{(1)}$ whose edges are $[ij]$
where $i,j \in U_k$. Due to $k \leq [\frac{n}{2}]$,  $U_k \cap V_0
\neq \emptyset$,  and $G_k$ is topologically an arc. By the
maximum principle applied to $g$ on the subgraph whose edges
consist of $[ij]$ with $i, j \in \{v \in V| d_c(v, u) \geq k\}$,
we obtain $g_A \leq \max_{i \in U_k} g_i$. Let $v_k \in U_k$ such
that $g_{v_k}=\max_{i \in U_k} g_i$  and
edge path $E_k$ be the shortest edge path in $G_k$ joining $v_k$ to a point $u_{k}$ in
$V_0-\{u\}$.  By construction $g_{u_{k}}=0$. Since $U_k$ contains at
most $3k+1$ vertices, the length  of $E_k$ is at most $3k$. The
Dirichlet energy $\mathcal E(g)$ of $g$ on $\T^{(1)}$ is given by
\begin{equation}\label{eq8} \mathcal E(g) =\frac{1}{2}\sum_{i \sim j}
\eta_{ij}(g_i-g_j)^2
 \geq
\sum_{k=1}^{[\frac{n}{2}]}
\mathcal E_k,$$ where
$$\mathcal E_k=\frac{1}{2}\sum_{[ij]\in \bar{E_k}}\eta_{ij}
(g_i-g_j)^2,\end{equation} and $\bar{E_k}$ be the set of oriented
edges in $E_k$.  Suppose $w_0=v_k \sim w_1 \sim w_2 \sim ... \sim
w_{l_k}=u_k$ are the vertices in the edge path $E_k$ where $l_k
\leq 3k$. Using the Cauchy-Schwartz inequality, we obtain
\begin{equation}\label{eq9} \mathcal
E_k=\sum_{i=1}^{l_k}\eta_{w_iw_{i-1}}
(g_{w_{i}}-g_{w_{i-1}})^2
\end{equation} $$\geq \frac{1}{M}\sum_{i=1}^{l_k} (g_{w_{i}}-g_{w_{i-1}})^2$$
$$\geq \frac{1}{M l_k} [\sum_{i=1}^{l_k}
(g_{w_{i}}-g_{w_{i-1}})]^2$$
$$\geq \frac{1}{3kM} (g_{v_k}-g_{u_{k}})^2
=\frac{g_{v_{k}}^2}{3kM} \geq \frac{g_A^2}{3kM}.$$

By (\ref{eq8}) and (\ref{eq9}), we obtain
\begin{equation}\label{eq10} \mathcal E(g) \geq
\frac{g_A^2}{3M}\sum_{k=1}^{[\frac{n}{2}]} \frac{1}{k} \geq
\frac{g_A^2 \ln(n)}{100M}.
\end{equation}
On the other hand, the Dirichlet principle says $\mathcal E(g)
=\min_{ h \in \R^V}\{ \frac{1}{2}\sum_{i \sim
j}\eta_{ij} (h_i-h_j)^2| h_u=1, h|_{V_0-\{u\}}=0\}.$ Take
$h\in \R^V$ to be $h_u=1$ and $h_i=0$ for all $i \in V-\{u\}$. We
obtain
$$\mathcal E(g) \leq \frac{1}{2}\sum_{i \sim j} \eta_{ij}(h_i
-h_j)^2 \leq 4M.$$ Combining this with (\ref{eq10}), we obtain
$$\frac{g_A^2 \ln(n)}{100M} \leq 4M,$$
i.e., $$g_A \leq \frac{20M}{\sqrt{ \ln(n)}}.$$
\end{proof}

\begin{proof}[Proof of Theorem \ref{1.2}]
We construct the smooth family $w(t) \in \R^V$ by
solving the system of ordinary differential equations (\ref{ode3})
where $(S, \T, l)=(\Delta ABC, \T, l_{st})$, $K^*|_{V-V_0\cup
\{A\}}=0$, $K^*_A=\pi-\alpha$ and $w_i(t)=0$ for $i\in V_0$.  By
the assumption that $\theta^i_{jk}(0) =\frac{\pi}{3}$ (i.e., $\T$
is an equilateral triangulation), $0\in W$ where the space $W$ is defined by
(\ref{wdomain}). By Lemma \ref{odeexist}, there exists a maximum
$s>0$ such that  a solution $w(t)$ to (\ref{ode3}) exists and
condition (d) holds for all $t \in [0, s)$.   We claim that $s
\geq 1$, $w(1)$ exists and $w(1)*l_{st}$ is a PL metric. In
particular, $w(1)*l_{st}$ satisfies condition (d) and $w(1) \in
W$.  Without loss of generality, let us assume that $s<\infty$. By
lemma \ref{odeexist} and condition (d),  we  obtain the following
two conclusions:
\begin{equation}\label{maxcon} \text{ $ \liminf_{t \to s^-} \eta_{ij}(w(t))=0$ for some $[ij]$,
\quad or $ \limsup_{ t \to s^-}
|\theta^i_{jk}(w(t))-\frac{\pi}{3}|= |\alpha-\frac{\pi}{3}|$ for
some $\theta^i_{jk}$.}
\end{equation}
The conclusion $\liminf_{t \to s^-} \theta^i_{jk}(w(t))=0$ is
ruled out by condition (d) which implies $\theta^i_{jk}(w(t)) \geq
\frac{\pi}{6}$.


We prove the claim that $s \geq 1$ as follows.
Since $\alpha \in
[\frac{\pi}{6}, \frac{\pi}{2}]$, we have $\frac{\pi}{3}+|\alpha
-\frac{\pi}{3}| \leq \frac{\pi}{2}$ and $\frac{\pi}{3}-|\alpha
-\frac{\pi}{3}| \geq \frac{\pi}{6}$. This shows, by $(d)$,
\begin{equation}\label{nondeg}\theta^i_{jk}(t) \in
[\frac{\pi}{6}, \frac{\pi}{2}] \quad \text{for all $t\in [0,s)$}.
\end{equation} In particular, $\cot(\theta^k_{ij})\geq 0$ and $\eta_{ij} \geq
\cot(\theta^k_{ij}) \geq 0$. Hence by definition we have
$$ |(\triangledown w')_{ij}|=\eta_{ij}|w_i'-w'_j| \geq
\ \cot(\theta^k_{ij})|w_i'-w_j'|.$$

By Lemma \ref{1.3} and the variation formula (\ref{curvatureevol})
that $\frac{dK_i}{dt}= (\Delta w')_i$, we obtain
$$2|(\triangledown w')_{ij}| \leq |(\Delta
w')_A|=|\frac{dK_A}{dt}| =|\alpha-\frac{\pi}{3}|.$$ This implies,
by (\ref{893}), the following,
\begin{equation}\label{ang-var} |\frac{ d\theta^k_{ij}}{dt}| \leq
\cot(\theta^i_{jk})|w_j'-w_k'| +\cot(\theta^j_{ik})|w_i'-w_k'|
\leq |(\triangledown w')_{jk}|+|(\triangledown w')_{ik}| \leq
 |\alpha -\frac{\pi}{3}|.\end{equation}

Therefore, for all $t \in [0, s)$,
\begin{equation}\label{18}|\theta^k_{ij}(t) -\frac{\pi}{3}|
=|\theta^k_{ij}(t) -\theta^k_{ij}(0)| \leq \int_0^t
|\frac{d\theta^k_{ij}(t)}{dt}|dt \leq t|\alpha-\frac{\pi}{3}| \leq
s|\alpha-\frac{\pi}{3}|.\end{equation} The above inequality shows
that $s \geq 1$. Indeed, if otherwise that $s<1$, using
(\ref{18}), we conclude that $\theta^i_{jk}(t) \in
[\frac{\pi}{3}-s|\alpha -\frac{\pi}{3}|, \frac{\pi}{3}+s|\alpha
-\frac{\pi}{3}|]$. In particular, $\liminf_{t \to s^-}
\eta_{ij}(t) \geq \cot(\frac{\pi}{3}+s|\alpha-\frac{\pi}{3}|)>0$
and $\limsup_{t \to s^-} |\theta^k_{ij}(t)-\pi/3| < |\alpha
-\pi/3| $. This contradicts (\ref{maxcon}).

 To see part (e), by (\ref{18}), if
$t \in [0,\frac{1}{2}]$, we have
$$|\theta^i_{jk}(t)-\frac{\pi}{3}| \leq
\frac{1}{2}|\alpha-\frac{\pi}{3}| \leq \frac{\pi}{12}, \quad
\text{i.e.,} \quad \theta^i_{jk}(t) \in [\frac{\pi}{4},
\frac{5\pi}{12}].$$

Now if $[ij]$ is an interior edge, then for $t\in[0,\frac{1}{2}]$
\begin{equation}\label{es5} |(\triangledown w')_{ij}|
=(\cot(\theta^k_{ij})+\cot(\theta^l_{ij}))|w'_i-w'_j|
\end{equation}
$$\geq
(1+\frac{\cot(\theta^l_{ij})}{\cot(\theta^k_{ij})})\cot(\theta^k_{ij})|w_i'-w_j'|$$
$$\geq (1+\cot(\frac{5\pi}{12}))\cot(\theta^k_{ij})|w_i'-w_j'|$$
$$ \geq \frac{5}{4} \cot(\theta^k_{ij})|w_i'-w_j'|.$$

If  $\theta^i_{jk}$ is an angle with $i \neq A$,  then either one
of the two edges $[ij]$, $[ik]$ is an interior edge, or  $i
\in\{B, C\}$. In the first case, say $[ij]$ is an interior edge,
using (\ref{es5}) and Lemma \ref{1.3}, for $t \in [0,1/2]$, we
have
\begin{equation}\label{eq38} |\frac{d \theta^i_{jk}}{dt}| \leq
\cot(\theta^k_{ij})|w'_i-w'_j| +\cot(\theta^j_{ik})|w_i'-w_k'|
\end{equation}
$$\leq \frac{4}{5} |(\triangledown w')_{ij}|+|(\triangledown
w')_{ik}|$$
$$\leq (\frac{4}{5}+1)\frac{|(\Delta w')_A|}{2} = \frac{9}{10} |\alpha
-\frac{\pi}{3}| \leq \frac{9}{10} \cdot
\frac{\pi}{6}=\frac{3\pi}{20}.$$

 In the second case that $i \in \{B,C\}$, one of the edges $[ij]$ or
 $[ik]$, say $[ij]$ is in the edge $BC$ of $\Delta ABC$,  i.e.,
 $w'_i=w'_j=0$. Therefore by Lemma \ref{1.3}, for $t \in
 [0,1/2]$, we have
\begin{equation}\label{eq39}|\frac{d \theta^i_{jk}}{dt}| \leq
\cot(\theta^k_{ij})|w'_i-w'_j| +\cot(\theta^j_{ik})|w_i'-w_k'|
\leq |(\triangledown w')_{ik}|$$ $$ \leq \frac{|(\Delta w')_A|}{2}
= \frac{1}{2} |\alpha -\frac{\pi}{3}| \leq \frac{3\pi}{20}.
\end{equation}

%

Therefore  if $\theta^i_{jk}$ is not the angle at $A$ and $t\in
[0,1)$, by (\ref{eq38}) and (\ref{eq39}), we have

$$ |\theta^i_{jk}(t) - \frac{\pi}{3}|=|\theta^i_{jk}(t)-\theta^i_{jk}(0)|\leq
\int_0^t|\frac{d \theta^i_{jk}}{dt}| dt \leq  \int_0^1|\frac{d
\theta^i_{jk}}{dt}|dt  = \int_0^{1/2}|\frac{d
\theta^i_{jk}}{dt}|dt + \int_{1/2}^1|\frac{d
\theta^i_{jk}}{dt}|dt$$
$$ \leq \frac{1}{2} \cdot \frac{3\pi}{20}
+\frac{1}{2} |\alpha -\frac{\pi}{3}| \leq
 \frac{3\pi}{40} +\frac{1}{2}\cdot
\frac{\pi}{6}=\frac{19\pi}{120}.$$ Therefore, $\theta^i_{jk}(t)
\in [\frac{21 \pi}{120}, \frac{59\pi}{120}]\subset (\frac{\pi}{6},
\frac{\pi}{2})$ for all $t\in [0, 1)$. Since conditions (d) and
(e) hold for all $t \in [0, 1)$, by definition of $\eta_{ij}$, we
see $\liminf_{ t \to 1} \eta_{ij}(w(t)) >0$.  Now we prove that
$w(1)$ is defined and $w(1)*l_{st}$ is a PL metric.  By the
estimates above, there exists $\delta>0$ such that for all $t \in
[0,1)$, $w(t) \in \mathcal W_{\delta}=\{ w \in W| \theta^i_{ij}
\geq \delta, \eta_{ij} \geq \delta\}$.  By Lemma \ref{odeexist},
the maximum time $t_0$ for which $w(t)$ exists on $[0, t_0)$ must
be greater than $1$. Therefore, $w(1)$ exists and $w(1) \in W$.
Since (d) and (e) are closed conditions, it follows that
$w(1)*l_{st}$ satisfies (d) and (e).




Now we prove part (f). By parts (d) and (e), we have
$\theta^i_{jk}(t) \in [\frac{\pi}{6}, \frac{59\pi}{120}]$ for $i
\neq A$ and $\theta^A_{jk} \in [\frac{\pi}{6}, \frac{\pi}{2}]$.
Since the conductance $\eta_{ij}$ is either $\cot(\theta^k_{ij})$
or a sum $\cot(\theta^k_{ij})+\cot(\theta^l_{ij})$, we obtain for
all edges $[ij]$ in $\T$, $ \eta_{ij}(t) \in [\cot(\frac{59
\pi}{120}), 2 \cot(\frac{\pi}{6})] \subset [\frac{1}{100}, 100]$.
Let $K_i(t)$ be the curvature of the metric $w(t)*l_{st}$ at the
vertex $i$. By Lemma \ref{891} for $f=\frac{1}{|\alpha
-\pi/3|}\frac{d w(t)}{dt}$ and $M=100$, we conclude that for all
$i \in V_0$,
$$|\frac{dK_i(t)}{dt}|=|(\Delta w')_i| \leq \frac{2000|\alpha
-\pi/3|}{\sqrt{\ln(n)}} \leq \frac{2000}{\sqrt{\ln(n)}}.$$
Therefore, $|K_i(t)-K_i(0)| \leq \int^t_0 |\frac{dK_i(t)}{dt}| dt
\leq \int^1_0 |\frac{dK_i(t)}{dt}| dt \leq
\frac{2000}{\sqrt{\ln(n)}}.$

Finally to prove (\ref{totalcur}), if $\alpha=\pi/3$, then all
$w(t)=0$ and $K(t)=K(0)$ and the result follows.  If $\alpha \neq
\pi/3$, we first claim that $w'_A(t) \neq 0$ for each $t$. Indeed,
if otherwise that $w_A'(t_1)=0$ for some $t_1$, then by the
maximum principle applied to the Dirichlet problem: $(\Delta
w'(t_1))_i=0$ for $i \in V-\{A\}\cup V_0$  and $w_i'(t_1)=0$ for
$i \in V_0\cup\{A\}$, we conclude $w_i'(t_1)=0$ for all $i\in V$.
In particular, $\alpha-\pi/3=(\Delta w')_A=0$ at $t=t_1$ which is
a contradiction. Therefore $w_A'(t) \neq 0$ and by the maximum
principle again $w_A'(t)w_i'(t) \geq 0$.
 Now if $i \in V_0$, then $K_i'(t)=\sum_{j \sim
i} \eta_{ji}(w_i'-w_j')=-\sum_{j \sim i} \eta_{ji} w_j'$. Since
$\eta_{ij} \geq 0$, therefore $w_A'(t)K_i'(t) \leq 0$ for $i \in
V_0$. It follows that for  all $i \in V_0$,
$(K_i(t)-K_i(0))w'_A(t) \leq 0$.  At the vertex $A$,
$|K_A(t)-K_A(0)|= | t(\alpha -\frac{\pi}{3})| \leq \frac{\pi}{6}.$
Therefore by the Gauss-Bonnet theorem that $K_A(t)+\sum_{ i \in
V_0} K_i(t)=K_A(t)+\sum_{i \in V} K_i(t)=2\pi$ and that
$K_i(t)-K_i(0)$ have the same signs for $i \in V_0$, we obtain $
\sum_{i \in V_0}|K_i(t)-K_i(0)|= |\sum_{i \in V_0} (K_i(t)-K_i(0))
|= |K_A(t)-K_A(0)|\leq \frac{\pi}{6}.$

\end{proof}

\subsection{A gradient estimate of discrete harmonic functions}

The proof Theorem \ref{1.1} is based on the following estimate.
Given a triangulated surface $(S, \T)$,  $v \in V(\T)$ and $r>0$,
we use $B_r(v)=\{ j \in V(\T) | d_c(j,v) \leq r\}$ to denote
combinatorial ball of radius $r$ centered at the vertex $i$ where $d_c$ is
the combinatorial distance on $\T^{(1)}$.

\begin{proposition} \label{311} Suppose $(\mathcal P, \T', l)$ is
polygonal disk  with an equilateral triangulation and
 $\T$ is the $n$-th standard subdivision of the triangulation
$\T'$ with $n \geq e^{10^6}$. Let $\eta: E=E(\T) \to [\frac{1}{M},
M]$ be a conductance function and $\Delta: \R^V \to \R^V$ be the
associated Laplace operator.  Let
 $V_0 \subset V(\T)$ be a thin subset such that for all $v \in V$ and
$m \leq n/2$, $|B_m(v) \cap V_0| \leq 10m$.   If $f: V \to \R$
satisfies $(\Delta f)_i=0$ for $i \in V-V_0$, $|(\Delta f)_i| \leq
\frac{M}{\sqrt{\ln(n)}}$ for  $i \in V_0$ and $\sum_{i \in V_0}
|(\Delta f)_i| \leq M$, then for all edges $[uv]$ in $\T$,
$$|f_u-f_v| \leq \frac{200M^3}{\sqrt{\ln(\ln(n))}}.$$
\end{proposition}

\begin{proof} Fix an edge $[uv]$ in the triangulation $\T$. Construct a function $g: V=V(\T) \to \R$ by solving  the
Dirichlet problem $(\Delta g)_i=0$ for $i \neq u,v$, and $g_u=1,
g_v=0$. By the maximum principle, we have $0 \leq g_i \leq 1$. By
the identity $\sum_{i\in V} (\Delta g)_i=0$ and that $g$ is not a
constant, we obtain $(\Delta g)_u=-(\Delta g)_v \neq 0$. Using the
Green's identity that $ \sum_{i \in V} (f_i (\Delta g)_i - g_i
(\Delta f)_i) =0$ and the assumptions of $f,g$, we obtain
$$ f_u (\Delta g)_u+f_v (\Delta g)_v -\sum_{i \in V_0} g_i (\Delta
f)_i=0.$$  Since $(\Delta g)_v=-(\Delta g)_u$, this shows
$$ f_u-f_v =\frac{1}{(\Delta g)_u}\sum_{ i \in V_0} g_i (\Delta
f)_i.$$ On the other hand, by the maximum principle $ g_u -g_j
\geq 0$, we have $ |(\Delta g)_u|=|\sum_{j \sim u}
\eta_{ju}(g_j-g_u)| = \sum_{j \sim u} \eta_{ju}(g_u -g_j) \geq
\frac{1}{M} (g_u-g_v) =\frac{1}{M}$. Therefore,
\begin{equation}\label{890} |f_u-f_v| \leq M |\sum_{i \in V_0} g_i (\Delta
f)_i|. \end{equation}

To estimate the right-hand side of (\ref{890}), take $r=[\sqrt[3]{\ln(n)}]$ and select $a
\notin B_r(u)$. Then using $0=\sum_{i \in V} (\Delta f)_i=\sum_{i
\in V_0} (\Delta f)_i$, $|g_i|\leq 1$, (\ref{890}) and the Lemma
\ref{318} below, we obtain
$$ |f_u-f_v| \leq M|\sum_{i \in V_0} g_i (\Delta f)_i| =  M|\sum_{i \in V_0} (g_i-g_a) (\Delta f)_i|   \leq
M\sum_{i \in V_0} |(g_i -g_a)||(\Delta f)_i|$$
$$\leq M ( \sum_{i \in V_0 \cap B_r(u)} |g_i-g_a| |(\Delta
f)_i|+ \sum_{ i \in V_0-B_r(u)} |g_i-g_a||(\Delta f)_i|)$$
$$\leq M (\frac{2M}{\sqrt{\ln(n)}}|V_0\cap B_r(u)| +
\frac{100M}{\sqrt{\ln(r)}} \sum_{i \in V_0} |(\Delta f)_i|)$$
$$\leq M[ \frac{20 M \sqrt[3]{\ln(n)}}{\sqrt{\ln(n)}}
+\frac{100M^2}{\sqrt{\ln (\sqrt[3]{\ln(n)})}}]$$
$$\leq \frac{200M^3}{\sqrt{\ln(\ln(n))}}.$$

In the last two steps, we have used $|V_0\cap B_r(u)|\leq 10r
=10\sqrt[3]{\ln n}$ and  $n \geq  e^{10^6}$ to ensure
$\frac{1}{\sqrt{\ln(\sqrt[3]{\ln(n)})}} \geq
\frac{\sqrt[3]{\ln(n)}}{\sqrt{\ln(n)}}.$
\end{proof}

\begin{lemma} \label{318}
Assume $(\mathcal P,\mathcal T',l), \mathcal T,E,M,\eta$ and $\Delta$ are as given in Proposition \ref{311}, and
$g$ is as given in the proof of Proposition \ref{311}, i.e., $(\Delta g)_i=0$ for $i \neq u,v$, and $g_u=1,
g_v=0$.
If $100\leq r \leq \frac{n}{3}$ and $\{a,b\} \cap
B_r(u)=\emptyset$, then
$$ |g_a -g_b| \leq \frac{100M}{\sqrt{\ln (r)}}.$$
\end{lemma}
The strategy of
the proof to Lemma \ref{318} is similar to that of Lemma \ref{891}.
\begin{proof}
For $k\leq r/3$, let $U_k=\{ i \in V |d_c(i,u) =k\}$. 
Since $\T$ is an equilateral triangulation of a flat surface, we
have $|U_k| \leq 6k$. Recall that a subset $U$ of $V=V(\T)$ is
called connected if any two points in $U$ can be joint by an edge
path in $\T^{(1)}$ whose vertices are in $U$. Each subset $U
\subset V$ is a disjoint of connected subsets which are called
connected components of $U$.   We claim that there exists a
connected component $G_k$ of $U_k$ such that $\{a,b\}$ lie in a
connected components of $V-G_k$.
 To see this, note that since $\T$ is the $n$-th standard
subdivision of $\T'$,  for all $k \leq r/3 \leq n/9$, the set
$B_k(u) =\{i\in V| d_c(i,u)\leq k\}$ is connected and $B_k(u)^c=\{
i \in V| d_c(i,u) >k\}$ has at most two connected components which
are also connected components of $V-U_k$. If $B_k(u)^c$ is
connected, then $U_k$ is connected and we take $G_k=U_k$.
If $B_k(u)^c$ has two connected components $R_1$ and $R_2$, then
there exists a non-flat boundary vertex $v' \in R_2$ such that
$d_c(u,v')\leq 3k \leq r$. 
This shows that $v' \in B_r(u)$. 
See Figure \ref{pp5}. The component $R_2$ is contained in $B_r(u)$
due to $d_c(v',u)\leq r$. Since $a, b \notin B_r(u)$, it follows
that $a,b$ are in $\mathcal R_1$. We take $G_k$ to be the
connected component of $U_k$ such that $R_1$ is a connected
component of $V-G_k$. Therefore, the claim follows.

\begin{figure}[ht!]
\begin{center}
\begin{tabular}{c}
\includegraphics[width=0.45\textwidth]{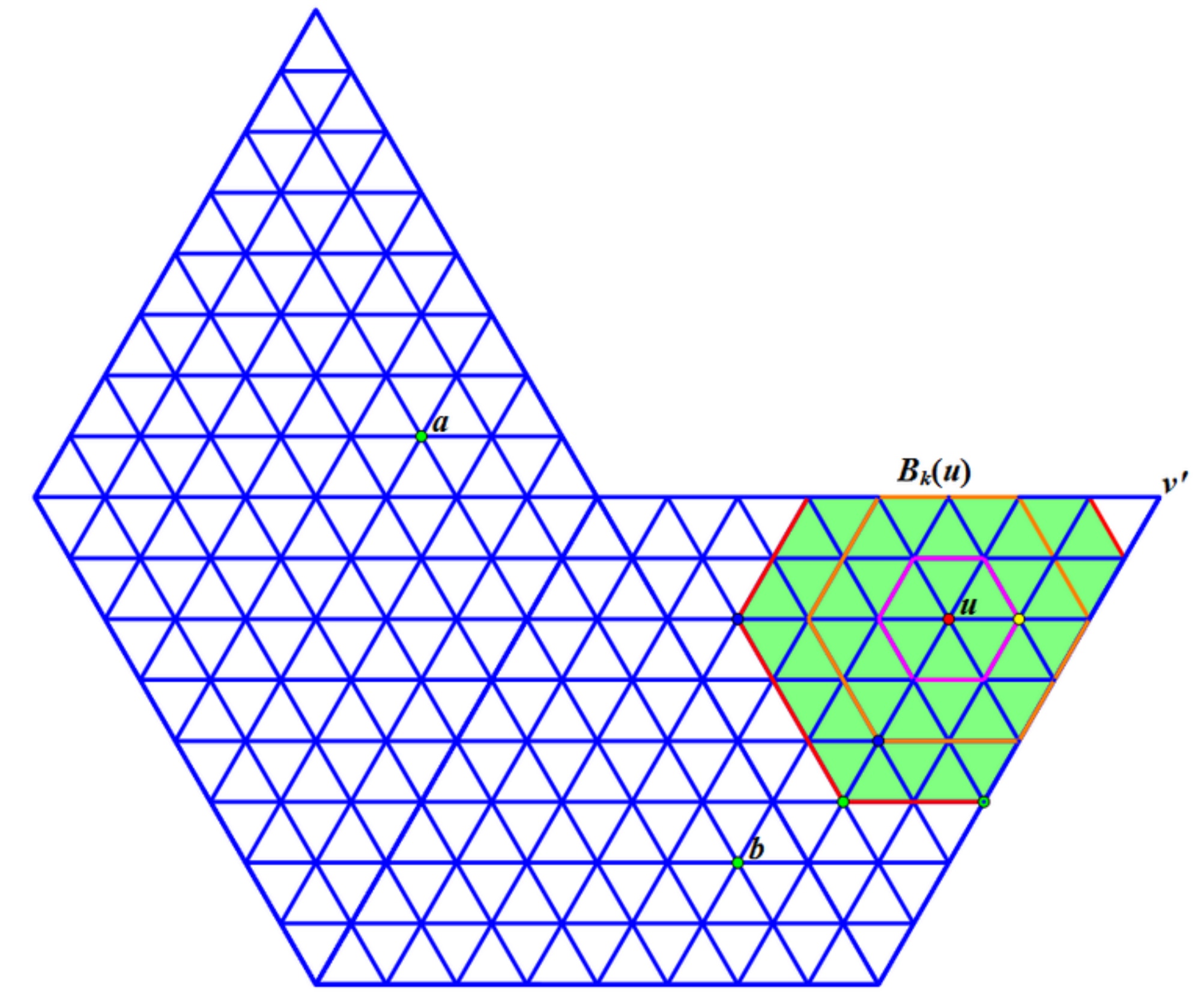}
\end{tabular}
\end{center}
\caption{Triangulated polygonal disks} \label{pp5}
\end{figure}


Let us assume without loss of generality that $g_a \leq g_b$. By
the maximum principle applied to $g$ on the connected graph whose
vertex set is the connected component of $V-G_k$ containing
$\{a,b\}$, there exist two vertices $u_k, u_k' \in G_k$ so that
$$ g_{u_k} \geq g_b, \quad \text{and} \quad g_{u'_k} \leq g_a.$$

Let $E_k$ be the shortest edge path with vertices in $G_k$
connecting $u_k$ to $u'_k$ and $\bar{E_k}$ be the set of all
oriented edges in $E_k$. The length of $E_k$ is at most $|G_k|\leq
6k$. The Dirichlet energy of $g$ on the graph $\T^{(1)}$ is
\begin{equation}\label{52}
\mathcal E(g) =\frac{1}{2}\sum_{i \sim j}
\eta_{ij}(g_i-g_j)^2 \geq
\\ \frac{1}{2M}\sum_{i \sim j} (g_i -g_j)^2 \geq \\
\frac{1}{2M}\sum_{k=1}^{[\frac{r}{3}]} \sum_{ [ij]\in \bar{E_k}}
(g_i-g_j)^2.
\end{equation}
Suppose $w_0=u_k \sim w_1 \sim w_2 \sim...\sim w_{l_k} =u'_k $ is
the edge path $E_k$ where $l_k \leq 6k$. Then by the
Cauchy-Schwartz inequality, we have
\begin{equation}\label{53}\frac{1}{2}\sum_{[ij]\in \bar{E_k}} (g_i-g_j)^2 =
\sum_{i=1}^{l_k}(g_{w_i}-g_{w_{i-1}})^2 \geq \frac{1}{l_k
}
(\sum_{i=1}^{l_k}(g_{w_i}-g_{w_{i-1}}))^2 \geq
\frac{1}{l_k}(g_{u_k}-g_{u'_k})^2 \geq
\frac{(g_a-g_b)^2}{6k}.\end{equation} Combining (\ref{52}) and
(\ref{53}), we obtain
\begin{equation}\label{54}
\mathcal E(g) \geq \frac{1}{2M}
\frac{(g_a-g_b)^2}{6}\sum_{i=1}^{[\frac{r}{3}]} \frac{1}{k} \geq
\frac{(g_a-g_b)^2 \ln (r)}{100M}. \end{equation}

On the other hand by the Dirichlet principle we have $\mathcal
E(g) \leq \frac{1}{2}\sum_{i \sim j}
\eta_{ij}(h_i-h_j)^2$ for any $h \in \R^V$ such that
$h_u=1, h_v=0$. Take $h$ to be $h_u=1$ and $h_i=0, \forall i \in
V-\{u\}$. We obtain $\mathcal E(g) \leq \frac{1}{2}\sum_{i \sim j}
\eta_{ij}(h_i-h_j)^2 \leq 6M$. Therefore,
$\frac{(g_b-g_a)^2 \ln(r)}{100M} \leq 6M$ which implies
$$|g_b-g_a| \leq \frac{100M}{\sqrt{\ln(r)}}.$$
\end{proof}

\subsection{A proof of Theorem \ref{1.1}}
For simplicity, a boundary vertex $v \in \mathcal P -\{p,q,r\}$
with non-zero curvature will be called a \it corner\rm. Note that
corners in $\T$ and its $n$-th standard subdivision $\T_{(n)}$ are
the same.  In particular, the total number of corners is
independent of $n$. Let $V_c$ be the set of all corner vertices.
Since $\mathcal P$ is embedded in $\C$, given a corner $v \in
V_c$, the degree $m$ of $v$ has to be 3,5 or 6. Consider the
combinatorial ball $B_{[n/3]}(v)$ of radius $[n/3]$ centered at a
corner $v \in V_c$ in $\T_{(n)}$. By construction $B_{[n/3]}(v)
\cap B_{[n/3]}(v') =\emptyset$ for distinct corners $v,v'$. Each
$B_{[n/3]}(v) $ is a union of $m-1$ $[n/3]$-th standard subdivided
equilateral triangles $\Delta_1, ..., \Delta_{m-1}$ in $\T$.
Applying Theorem \ref{1.2} with $\alpha=\frac{\pi}{m-1}$ to the
triangulated  equilateral triangle $(\Delta_i, v)$ for each
$i=1,2,...,m-1$, we produce a discrete conformal factor
$w(\Delta_i)\in \R^{V(\Delta_i)}$ for each $\Delta_i$ such that if
a vertex  $u \in V(\Delta_i)\cap V(\Delta_j)$, then
$w_u(\Delta_i)=w_u(\Delta_j)$. In particular there is a well
defined discrete conformal factor $w(B_{[n/3]}(v))$ on
$B_{[n/3]}(v)$ obtained by gluing these $w(\Delta_i)$. See Figure
\ref{p4}.
 Define $w^{(1)}:
V(\T_{(n)}) \to \R$ as follows: if $u \in \cup_{v \in V_c}
B_{[n/3]}(v)$, then $w^{(1)}_u=w_u(B_{[n/3]}(v))$ for $u \in
B_{[n/3]}(v)$ and $w^{(1)}(u)=0$ for  $u \notin \cup_{v \in V_c}
B_{[n/3]}(v)$. Let $\hat{l}=w^{(1)}*l_{st}$ be the PL metric on
$\T_{(n)}$ and  $\hat{K}$ be its the discrete curvature. Let $K^*:
V_{(n)} \to \R$ be defined by $K^*_i=0$ if $i \notin\{p,q,r\}$,
and $K^*_i=\frac{2\pi}{3}$ if $i \in \{p,q,r\}$.
 By Theorem \ref{1.2}, the PL metric $\hat{l}$ and $\hat{K}$ satisfy
the following:

\noindent (a) the curvature $\hat{K}_i=K^*_i$  at all vertices $i$ such
that $d_c(i,v)\neq[n/3]$ for some corner $v \in V_c$;
\\
(b) $w^{(1)}_i=0$ for $i \notin \cup_{v \in V_c} B_{[n/3]}(v)$
;\\
(c) all inner angles at a corner $v \in V_c$  are in $[\frac{\pi}{6}, \frac{\pi}{2}]$;\\
(d) all inner angles at a non-corner vertex  are  in  $[\frac{\pi}{6}, \frac{59\pi}{120}]$;\\

\noindent (e) $|\hat{K}_i-K^*_i| \leq\frac{4000}{\sqrt{\ln(n)}}$
 and $\sum_{i
\in V} |\hat{K}_i-K^*_i|\leq \frac{\pi N}{3}$ where $N$ is the
number of corners in $\mathcal P$.

We will find a discrete conformal factor $w^{(2)}: V_{(n)} \to \R$
such that $w^{(2)}*\hat{l}$ and its curvature satisfy Theorem
\ref{1.1} by solving the following system of ordinary differential
equations in $w(t)$:
\begin{equation}\label{ode4}
 \frac{d K_i(w(t)*\hat{l})}{dt} =  K^*_i-\hat{K_i}, \forall i \in V(\T_{(n)})-\{p,q,r\};
 w_s(t)=0, s \in \{p,q,r\};
 \quad \text{and} \quad w(0)=0.
\end{equation}
Let $K(t)=K(w(t)*\hat{l})$.  Note that (\ref{ode4}) and the
Gauss-Bonnet formula imply that $K_p'(t)=K_p^*-\hat{K_p}$.
 By Lemma
\ref{odeexist}, the solution $w(t)$ exists on some interval
$[0, \epsilon)$. Our goal is to show that for $n$ large, the solution
$w(t)$ exists on $[0, 1]$. In this case, the conformal factor
$w^{(2)}$ is taken to be $w(1)$. The required discrete conformal
factor $w_n$ in Theorem \ref{1.1} is taken to be
$w^{(1)}+w^{(2)}$.

Consider the maximum time $t_0$ such that the solution $w(t)$ to
(\ref{ode4}) exists for $t\in [0, t_0)$ and the PL metrics
$w(t)*\hat{l}$ satisfy: \\
($c'$) all inner angles at a corner $v \in V_c$  are in
$[\frac{\pi}{6}-\frac{\pi}{1000},
\frac{\pi}{2}+\frac{\pi}{1000}]$;\\
($d'$) all inner angles at a non-corner vertex are in
$[\frac{\pi}{6}-\frac{\pi}{1000}
,\frac{59\pi}{120}+\frac{\pi}{1000}]$.\\

  Let $V_0=\cup_{v \in V_c} \{ i \in V_{(n)} | d_c(i, v)=[n/3]$\}. By
 construction,  $|B_r(i)\cap V_0|\leq
10r$ for all $r \leq n/3$.  Then
 $\sum_{i \in V_0} |(\Delta w'(t))_i| = \sum_{i \in
V_0} |K_i'(t)| \leq \sum_{i \in V_0} |\hat{K_i}-K_i^*| \leq
\frac{\pi N}{3}$ and  $|(\Delta w')_i| =|K_i'(t)| \leq
|\hat{K_i}-K^*_i| \leq \frac{4000}{\sqrt{\ln(n)}}$. Choose
$M=\max\{  4000, \frac{\pi N}{3}\}$.
Then by $(c')$, $(d')$, $(e)$ and the formula
$\cot(a)+\cot(b)=\frac{\sin(a+b)}{\sin(a)\sin(b)}$,  for all $t\in
[0, t_0)$,
 we have
 $\eta_{ij}(t)=\eta_{ij}(w(t)*\hat{l}) \in [\frac{1}{4000}, 4000]\subset
 [\frac{1}{M}, M]$,
$(\Delta w')_i=0$ for $i \in V(\T_{(n)})-V_0$, $|(\Delta w')_i|
\leq \frac{M}{\sqrt{\ln(n)}}$ and $\sum_{i \in V_0}|(\Delta w')_i|
\leq M$. In summary, $f=w'$ satisfies conditions in Proposition
\ref{311} for all $t \in [0, t_0)$.     By Proposition \ref{311},
if  $i \sim j$, then
$$ |w'_i(t)-w_j'(t)| \leq \frac{200 M^3}{\sqrt{\ln(\ln(n))}}.$$
On the other hand, by the variation of angle formula (\ref{893})
and $M \geq  |\cot(\theta^k_{ij})|$, we have
$$|\frac{d\theta^k_{ij}}{dt}| \leq
|\cot(\theta^i_{jk})(w'_j-w'_k)|+|\cot(\theta^j_{ik})(w_i'-w_k')|
\leq M (|w_j'-w_k'|+ |w_i'-w_k'|) \leq
\frac{400M^4}{\sqrt{\ln(\ln(n))}}.$$
 Therefore, for $t \in [0, t_0)$ and sufficiently large $n$,
\begin{equation}\label{ineq9}
|\theta^k_{ij}(w(t))-\theta^k_{ij}(0)| \leq \int_0^t |\frac{d
\theta^k_{ij}(w(t))}{dt}| dt \leq \frac{400M^4
t_0}{\sqrt{\ln(\ln(n))}} \leq \frac{\pi t_0}{2000}.
\end{equation}
It follows that $t_0 >1$ (or $t_0=\infty$) since otherwise, by
(\ref{ineq9}), the choices of angles in (c),(d), ($c'$), ($d')$
 and Lemma \ref{odeexist}, we can extend the solution $w(t)$ to
$[0, t_0+\epsilon)$ for some $\epsilon >0$ such that $(c')$ and
$(d')$ still hold. To be more precise, by Lemma \ref{odeexist} on
the maximality of $t_0$, we have either $\limsup_{t \to t_0}
|\theta^i_{jk}(t)-\frac{\pi}{3}|= \frac{\pi}{1000}$
for an inner
angle $\theta^i_{jk}$ at a corner $i \in V_c$,  or $\limsup_{t \to
t_0}\theta^i_{jk}(t)=\frac{\pi}{6}-\frac{\pi}{1000}$ or
$\limsup_{t \to t_0}\theta^i_{jk}(t)=\frac{59\pi}{120}+\frac{\pi}{1000}$
for
an angle $\theta^i_{jk}$ at a non-corner vertex $i$.
But, due to (\ref{ineq9}), none of these
conditions holds   if $t_0\leq 1$.
 Therefore the
solution  $w(1)$ exists. By construction, the curvature $K(1)$ of
$w(1)*\hat{l}$ is $K(0)+\int^1_0 K'(t)dt=\hat{K}+K^*-\hat{K}=K^*$.
Furthermore, condition (c) in Theorem \ref{1.1} follows from
($c'$) and ($d')$.

\section{A proof of the convergence theorem}

We will prove the
following theorem.

\begin{theorem}\label{conv1} Let $\Omega$ be a Jordan domain in the complex
plane and  $\{p,q,r\} \subset \partial \Omega$. There exists a
sequence of triangulated polygonal disks $(\Omega_{n}, \T_n,
d_{st}, (p_n, q_n, r_n))$ where $\T_n$ is an equilateral
triangulation and $p_n, q_n, r_n$ are three boundary vertices such
that

(a) $\Omega=\cup_{n=1}^{\infty} \Omega_n$ with $\Omega_n \subset
\Omega_{n+1}$, and  $\lim_n p_n =p$, $\lim_n q_n =q$ and $\lim_n r_n=r$,

(b)
 discrete uniformization maps $f_n$ associated to $(\Omega_{n},
\T_n, d_{st}, (p_n, q_n, r_n))$  exist and  converge uniformly to the
 Riemann mapping associated to
 $(\Omega, (p,q,r))$.
\end{theorem}






Before giving the proof, let us recall Rado's theorem and its generalization to quasiconformal maps.
If $\phi: \D \to \Omega$ is a K-quasiconformal map  onto a Jordan
domain $\Omega$, then $\phi$ extends continuously to a
homeomorphism $\overline{\phi}: \overline{\D} \to
\overline{\Omega}$ between their closures (see \cite[corollary on
page 30]{ahlf}). If $K=1$, $\overline{\phi}$ is the Caratheodory
extension of the Riemann mapping.  A sequence of Jordan curves
$J_n$ in $\C$ is said to converge uniformly to a Jordan $J$ curve
in $\C$ if there exist homeomorphisms $\phi_n: \mathbb S^1 \to
J_n$ and $\phi: \mathbb S^1 \to J$ such that $\phi_n$ converges
uniformly to $\phi$.  Rado's theorem
\cite{pomo} and its extension by Palka \cite[corollary 1]{palka}
states that,

\begin{theorem}[Rado, Palka] \label{rado} Suppose $\Omega_n$ is a sequence of
Jordan domains such that $\partial \Omega_n$ converges uniformly
to  $\partial \Omega$.  If $f_n: \D \to \Omega_n $ is a  K-quasiconformal map for each $n$ such that the sequence \{$f_n$\}
converges to a K-quasiconformal map $f: \D \to \Omega$ uniformly
on compact sets of $\D$, then $\overline{f_n}$ converges to
$\overline{f}$ uniformly on $\overline{\D}$.
\end{theorem}

The following compactness result is a consequence of Palka's
theorem (\cite[corollary 1]{palka}) and Lehto-Virtanen's work
\cite[Theorems 5.1, 5.5]{lv}.


\begin{theorem}\label{cor6} Suppose $\Omega_n$ is a sequence of
Jordan domains such that $\partial \Omega_n$  converges uniformly
to  $\partial \Omega$ and $K>0$ is a constant. Let $p_n, q_n, r_n
\in
\partial \Omega_n$ and $p,q,r \in
\partial \Omega$ be distinct points such that $\lim_{n}p_n=p,
\lim_n q_n =q, \lim_n r_n =r$ and $h_n: \D \to \Omega_n$ be
K-quasiconformal maps such that $\overline{h_n}$ sends $(1,
\sqrt{-1},  -1)$ to $(p_n, q_n, r_n)$. Then there exists a
subsequence $\{h_{n_i}\}$ of \{$h_n\}$ converging uniformly on
$\overline{\D}$ to a K-quasiconformal map $h: \D \to \Omega$
sending $(1,   \sqrt{-1}, -1)$ to $(p,q,r)$.
\end{theorem}

Now we prove Theorem \ref{conv1}.
\begin{proof}
Given a Jordan
domain $\Omega$ with three distinct points $p,q,r$ in $\partial
\Omega$, construct a sequence of approximating polygonal disks
$\Omega_n$ such that (1) each $\Omega_n$ is triangulated by
equilateral triangles of side lengths tending to $0$, (2)
$\partial \Omega_n$ converges uniformly to the Jordan curve
$\partial \Omega$ such that $\Omega_n \subset \Omega_{n+1}$, (3)
there are three boundary vertices $p_n,q_n, r_n \subset
\partial \Omega_n$ such that $\lim_n p_n =p$, $\lim_n q_n =q$ and
$\lim_n r_n =r$, and (4) the curvatures of $\Omega_n$ at $p_n,
q_n, r_n$ are $\frac{2\pi}{3}$ and curvatures of $\Omega_n$ at all
other boundary vertices are not $\frac{2\pi}{3}$.

By Theorem \ref{1.1}, we produce a  standard subdivision $\T_n$ of
$\Omega_n$ and $w_n \in \R^{V(\T_n)}$ such that $(\Omega_n, \T_n,
w_n*l_{st})$ is isometric to the equilateral triangle $(\Delta
ABC, \T_n', l_n')$ with a Delaunay triangulation $\T_n'$ and
$A,B,C$ correspond to $p_n,q_n,r_n$. Let $f_n: (\Delta ABC,
\T_n',$ $( A,B,C)) \to (\Omega_n,$$ \T_n, (p_n, q_n, r_n))$ be the
associated discrete conformal map and $\bar{f}: (\Delta ABC, $
$(A,B,C))$$\to (\overline{\Omega}, (p,q,r))$ be the
Riemann mapping. We claim that $f_n$ converges uniformly to
$\bar{f}$ on
$\Delta ABC$. 

To establish the claim, first by Theorem \ref{1.1}, we know all
angles of triangles in the triangulated PL surface $(\Delta ABC, \T_n', l_n')$  are
at least $\epsilon_0>0$.  Therefore the discrete conformal maps $f_n$ are
K-quasiconformal for a constant $K$ independent of $n$.  By Theorem \ref{cor6}, it follows that every
limit function $g$ of a convergence subsequence
 \{$f_{n_i}$\} is a K-quasiconformal map from $int(\Delta ABC)$
to $\Omega$ which extends continuously to $\Delta ABC$ sending $A,
B,C$ to $p,q,r$ respectively. We claim that the limit map $g$
is conformal. Indeed, by Lemma \ref{hex}, the discrete conformal
map $f_n^{-1}$, when restricted to a fixed compact set $R$ of
$\Omega$, maps equilateral triangles in $\T_n$ which are inside
$R$ to triangles of $\T'_n$ that become arbitrarily close to
equilateral triangle as $n\to \infty$. Therefore the limit
map $g$ of the subsequence $f_{n_i}$ is 1-conformal and
therefore conformal in $int(\Delta ABC)$. The continuous extension
of $g$ sends $A,B,C$ to $p,q,r$ respectively by Theorem
\ref{cor6}. On the other hand, there is only one Riemann mapping
$f: int(\Delta) \to \Omega$ whose continuous extension sends
$A,B,C$ to $p,q,r$ respectively.
 Therefore, $g=f$.
This shows all limits of convergence subsequences of \{$f_n$\} are
equal $f$. Therefore $\{f_n\}$ converges to $f$ uniformly on
compact sets in $int(\Delta ABC)$. By Theorem \ref{rado},
$\overline{f_n}$ converges uniformly to $\bar{f}$.
\end{proof}




\section{A convergence conjecture on discrete uniformization maps}

We  discuss  a general approximation conjecture and the related topics of discrete conformal equivalence of polyhedral metrics.

\subsection{A strong version of convergence of discrete conformal maps}

As discussed before, the main drawback of the vertex scaling operation on polyhedral metrics is the lacking of an existence theorem. For instance, given a PL metric on a closed triangulated surface $(S, \T, l)$,  there is in general no discrete conformal factor $w: V \to \R$ such that
the new PL metric $(S, \T, w*l)$ has constant discrete curvature.

The recent work of \cite{glsw} established an existence and a uniqueness theorem for polyhedral metrics by allowing the triangulations to be changed.

\begin{definition} \label{ddc}(Discrete conformality of PL metrics \cite{glsw})
 Two PL metrics $d, d'$ on $(S,V)$ are discrete conformal  if there exist sequences
 of PL
metrics $d_1=d, ..., d_m =d'$ on $(S, V)$ and triangulations
$\T_1, ..., \T_m$ of $(S, V)$  satisfying

(a) (Delaunay) each $\T_i$ is Delaunay in $d_i$,

(b)(Vertex scaling) if $\T_i=\T_{i+1}$, there exists a function
$w:V \to \R$ so that if $e$ is an edge in $\T_i$ with end points
$v$ and $v'$, then the lengths $l_{d_{i+1}}(e)$ and $l_{d_i}(e)$
of $e$ in $d_i$ and $d_{i+1}$ are related by
\begin{equation}\label{conformal}
l_{d_{i+1}}(e) = e^{ w(v)+w(v')} l_{d_i}(e) ,
\end{equation}

(c) if $\T_i \neq \T_{i+1}$, then $(S, d_i)$ is isometric to $(S,
d_{i+1})$ by an isometry homotopic to the identity in $(S, V)$.
\end{definition}

The main theorem proved in \cite{glsw} is the following.
\begin{theorem}
\label{unif} Suppose $(S, V)$ is a closed connected marked surface
and $d$ is a PL metric on $(S, V)$. Then for any $K^*:V \to
(-\infty, 2\pi)$ with $\sum_{v \in V} K^*(v) =2\pi \chi(S)$, there
exists a PL metric $d^*$, unique up to scaling and isometry
homotopic to the identity, on $(S, V)$ such that $d^*$ is discrete
conformal to $d$ and the discrete curvature of $d^*$ is $K^*$.
Furthermore, the metric $d^*$ can be found using a finite
dimensional (convex) variational principle.
\end{theorem}

There is a close relation between the discrete conformal equivalence defined in  Definition \ref{ddc} and convex geometry in hyperbolic 3-space.
The first work relating vertex scaling operation and hyperbolic geometry is in the paper by Bobenko-Pinkall-Springborn \cite{bps}. They associated each polyhedral metric on $(S, \T, l)$ a hyperbolic metric with cusp end on the punctured surface $S-V(\T)$.  However, the Delaunay condition on the triangulation $\T$ was missing in their definition.  The discrete conformal equivalence defined in Definition \ref{ddc} is equivalent to the following hyperbolic geometry construction.
Let $(S, V,d)$ be a PL surface. Take a Delaunay triangulation $\T$ of $(S, V, d)$ and consider the PL metric $d$ as isometric gluing of Euclidean triangles $\tau \in \T$.  Consider each triangle $\tau$ in $\T$  as the Euclidean convex hull of three points $v_1, v_2, v_3$ in the complex plane $\C$.  Let $\tau^*$ be the convex hull of $\{v_1, v_2, v_3\}$ in  the upper-half space model of the hyperbolic 3-space $\H^3$. Thus $\tau^*$ is an ideal hyperbolic triangle having the same vertices as that of $\tau$.  If $\sigma$ and $\tau$ are two Euclidean triangles in $\T$ glued isometrically along  two edges by an isometry $f $ considered as an isometry of the Euclidean plane, we glue $\tau^*$ and $\sigma^*$ along the corresponding edges using the \it same \rm map $f$ considered as an isometry of $\H^3$.  Here we have used the fact each isometry of the complex plane extends naturally to an isometry of the hyperbolic 3-space $\H^3$.  The result of the gluing of these $\tau^*$ produces a hyperbolic metric $d^*$ on the punctured surface $S-V$. It is easy to see  that $d^*$ is independent of the choices of Delaunay triangulations.
It is shown in \cite{glsw} (see also \cite{gglsw}) that two PL metrics $d_1$ and $d_2$ on $(S,V)$ are discrete conformal in the sense of Definition \ref{ddc} if and only if the associated hyperbolic metrics $d_1^*$ and $d_2^*$ are isometric by an isometry homotopic to the identity on $S-V$.

 Using this hyperbolic geometry interpretation, one defines the discrete conformal map between two discrete conformally equivalent PL metrics $d_1$ and $d_2$ as follows (see \cite{bps} and \cite{glsw}).  The vertical projection of the ideal triangle $\tau^*$ to $\tau$ induces a homeomorphism $\phi_d: (S-V, d^*)  \to (S-V, d)$.  Suppose $d_1$ and $d_2$ are two discrete conformally equivalent PL metrics on $(S,V)$. Then the {\it discrete conformal map}  from $(S, V, d_1)$ to $(S, V, d_2)$ is given by $\phi_{d_2} \circ \psi \circ \phi_{d_1}^{-1}$ where  $\psi: (S, V, d_1^*) \to (S, V, d_2^*)$ is the hyperbolic isometry. Note that in this new setting, discrete conformal maps are piecewise projective instead of piecewise linear.

Theorem \ref{unif} can be used for approximating Riemann mappings for Jordan domains.
 Given a
simply connected polygonal disk with a PL metric $(D, V, d)$ and
three boundary vertices $p,q,r \in V$, let the metric double of
$(D, V,d)$ along the boundary be the polyhedral 2-sphere $(\mathbb
S^2, V', d')$. Using Theorem \ref{unif}, one produces a new
polyhedral surface $(\mathbb S^2, V',
 d^*)$  such that: 11) $(\mathbb S^2, V', d^*)$ is discrete conformal to $(\mathbb
S^2, V', d')$; ( 2) the discrete curvatures of $d^*$ at $p,q,r$ are
$4\pi/3$;  (3) the discrete curvatures of $d^*$ at all other
vertices are zero; and  (4) the area of $(\mathbb S^2, V',
 d^*)$ is $\sqrt{3}/2$. Therefore $(\mathbb S^2, V',
 d^*)$ is isometric to the metric double
($\mathcal D(\Delta ABC), V'', d'')$ of an equilateral triangle
$\Delta ABC$ of edge length 1. Let  $F$  be the discrete conformal
map from ($\mathcal D(\Delta ABC), V'', d'')$  to $(\mathbb S^2,
V', d')$ such that $F$ sends $A,B,C$ to $p,q,r$ respectively. Due
to the uniqueness part of Theorem \ref{unif}, we may assume that
$f=F|: \Delta ABC \to D$ and $f$ sends $A,B,C$ to $p,q,r$
respectively. We call $f$ the \it discrete uniformization map \rm
associated $(D, V, d, (p,q,r))$.

A strong form of the convergence is the following,
\begin{conj}   Let $(\Omega$, $(p,q,r)$) be a Jordan domain in
the complex plane with three marked boundary points  and
$(\Omega_{n}, \T_n, d_{st},$ $ (p_n, q_n, r_n))$ be any sequence of
triangulated flat polygonal disks with three marked boundary
vertices such that

(a) $\T_n$ is an equilateral triangulation,

(b) $\partial \Omega_n$ converges uniformly to $\partial \Omega$,

(c)  the edge length of $\mathcal T_n$ goes to zero,

(d) $\lim_n p_n =p$, $\lim_n q_n =q$ and $\lim_n r_n=r$.

Then
discrete uniformization maps $f_n$ associated to $(\Omega_{n},
\T_n, d_{st}$,
$(p_n, q_n, r_n))$  converge uniformly to the
 Riemann mapping associated to
 $(\Omega, (p,q,r))$.
\end{conj}

\subsection{Discrete conformal equivalence and convex sets in the hyperbolic 3-space}
We now discuss  the relationship between discrete
conformal equivalence defined in Definition \ref{ddc}, ideal convex sets in the hyperbolic 3-space
$\H^3$ and the motivation for Conjectures \ref{c1} and \ref{c2}.

The classical uniformization theorem for Riemann surfaces follows
from the special case that every simply connected Riemann surface
is biholomorphic to $\C$, $\D$ or $\mathbb S^2$.  The discrete
analogous should be the statement that each non-compact simply
connected polyhedral surface is discrete conformal to either $(\C,
V, d_{st})$ or $(\D, V, d_{st})$ where $V$ is a discrete set and $d_{st}$ is the standard Euclidean metric. Furthermore, the set $V$ is unique up to M\"obius transformations.
For a non-compact polyhedral surface $(S, V, d)$ with an
infinite set $V$, the
 hyperbolic geometric view point of discrete conformality is a better approach.
Namely discrete conformal equivalence  between two PL metrics is
the same as  the Teichm\"uller equivalence between their
associated hyperbolic metrics.  For instance, if we take a Delaunay triangulation $\T$ of the complex plane $(\C, d_{st})$ with vertex set $V$, then the associated hyperbolic metric $d^*_{st}$ on $\C-V$ is isometric to the boundary of the convex hull $\partial C_{\H}(V)$ in $\H^3$.  Therefore, a PL surface $(S, V', d)$ is discrete conformal to $(\C, V, d_{st})$ for some discrete subset $V \subset \C$ if and only if the associated hyperbolic metric $d^*$ is  isometric to the boundary of the convex hull $\partial C_{\H}(V)$ .  It shows discrete
uniformization is the same as
realizing hyperbolic metrics as the boundaries of  convex hulls
(in $\H^3$) of closed sets in $\partial \H^3$.
 One can formulate the
conjectural discrete uniformization theorem as follows.
Given a discrete set $V'$ in $\C$ or $\D$, let $\hat{d}$
be the unique conformal complete hyperbolic metric on $\C-V'$ or
$\D-V'$. Then $\hat{d}$ is isometric to the boundary of the convex
hull of a discrete set $V \subset \C$ or $(\C\cup \{\infty\} -\D)
\cup V$ where $V$ is discrete and
unique up to M\"obius transformations. This is the original
motivation for proposing Conjectures \ref{c1} and \ref{c2}.



These two conjectures bring discrete uniformization close to the
classical Weyl problem on realizing surfaces of non-negative
Gaussian curvature as the boundaries of convex bodies in the
3-space.  In the hyperbolic 3-space $\H^3$, convex surfaces have curvature at least $-1$.
The  work of Alexandrov \cite{al} and Pogorelov
\cite{pogo} show that for each path metric $d$ on the 2-sphere
$\SS^2$ of curvature $\geq -1$, there exists a compact convex
body, unique up to isometry, in $\H^3$ whose boundary
is isometric to $(\SS^2, d)$.   
The interesting remaining cases are  non-compact surfaces of
genus zero in the hyperbolic 3-space $\H^3$.   A theorem of
Alexandrov \cite{al} states that any complete surface of genus
zero whose curvature is at least $-1$ is isometric to the boundary
of a closed convex set in $\H^3$. On the other hand, given a
closed set $X \subset \C$, W. Thurston proved that the intrinsic
metric on $\partial C_{\H}(X)$ is complete hyperbolic  (see \cite{em} for a proof). Putting these two theorems
together, one sees that each complete hyperbolic metric on a
surface of genus zero is isometric to the boundary of the convex
hull of a closed set in the Riemann sphere. However, in this generality, the uniqueness of the convex surface is false.
Conjectures \ref{c1} and \ref{c2} say that one has both the existence and uniqueness if
one imposes restricts to the  boundaries of the convex hulls of
closed sets.

There are some evidences supporting Conjectures \ref{c1} and \ref{c2}.  The work of  Rivin \cite{rivin2} and Schlenker \cite{sch1} show
that Conjectures \ref{c1} and \ref{c2} hold if $\Omega$ has finite area (i.e., $X$ is a finite set) or if $\Omega$ is conformal to the 2-sphere
with a finite number of disjoint disks removed (i.e., $X$ is a finite disjoint union of round disks). 
Our recent work \cite{lw} shows that Conjectures \ref{c1} holds
for $\Omega$ having countably many topological ends using the work of He-Schramm on K\"obe conjecture.

One should compare Conjectures \ref{c1} and \ref{c2} with the K\"obe circle domain
conjecture which states that each genus zero Riemann surface is
biholomorphic to the complement of a circle type closed set in the
Riemann sphere. The work of He-Schramm \cite{hes} shows that
K\"obe conjecture holds for surfaces with countably many ends and
the circle type set is unique up to M\"obius transformations.
Uniqueness is known to be false for the K\"obe conjecture in
general.  Our recent work \cite{lw} shows that the K\"obe conjecture is equivalent to  Conjecture \ref{c1}.
 Other related  works are  \cite{bo},  \cite{fila},  \cite{labourie}, \cite{lec},  \cite{rivin2},  \cite{sch1}, \cite{schlenker}, and \cite{spring}.


We end this paper by proposing the following the conjecture. The work
of Rodin-Sullivan \cite{RS} and Theorem \ref{rigid} show the
rigidity phenomena for the two most regular patterns (regular
hexagonal circle packing and regular hexagonal triangulation) in
the plane. These rigidity results can be used to approximate the
Riemann mappings and the uniformization metrics. The third regular
pattern in the plane is the hexagonal square tiling in which each
square of side length one interests exactly six others. See figure
\ref{90}.
\begin{figure}[ht!]
\begin{center}
\begin{tabular}{c}
\includegraphics[width=0.2\textwidth]{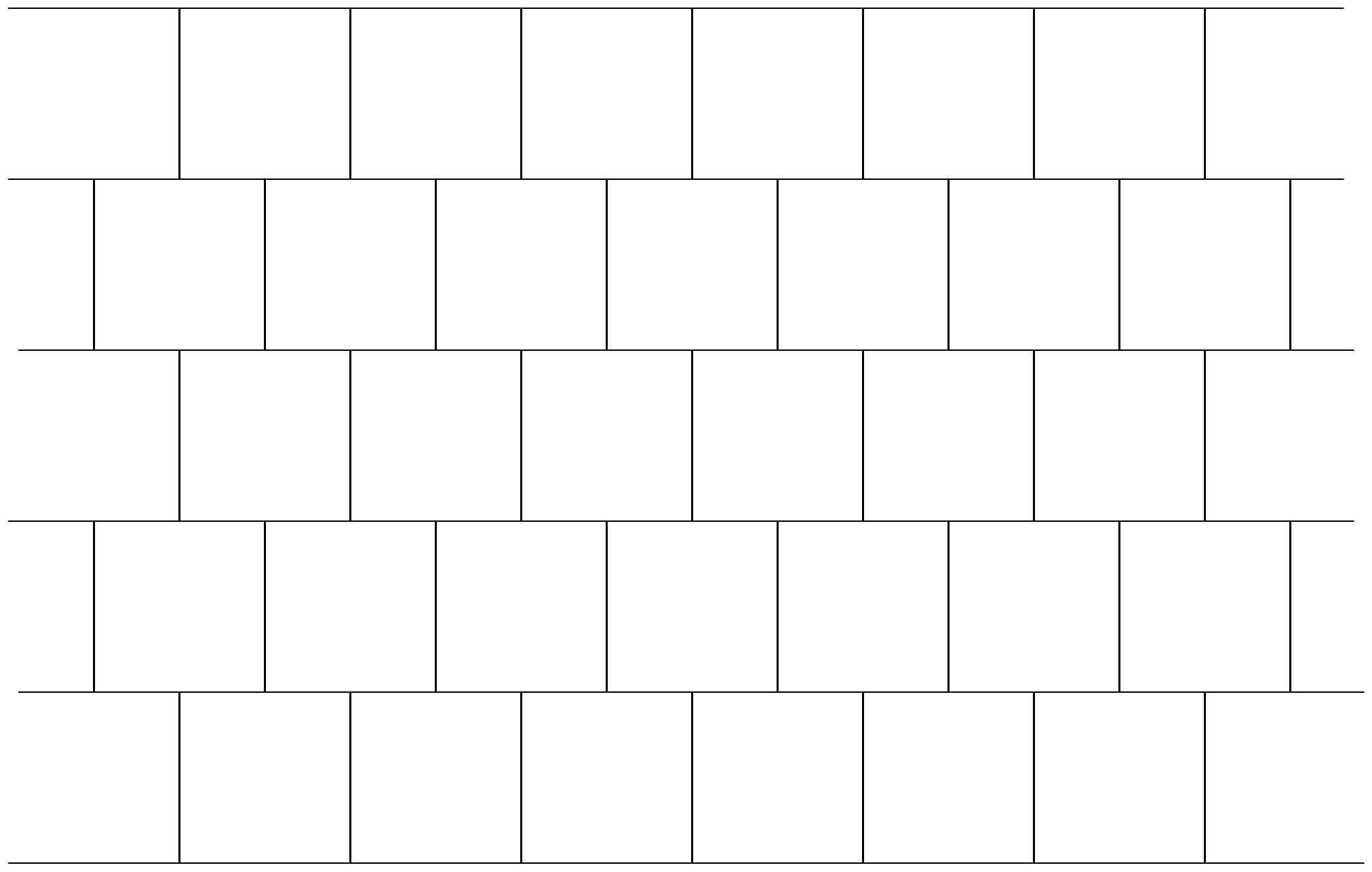}
\end{tabular}
\end{center}
\caption{Regular hexagonal square tiling} \label{90}
\end{figure}


\begin{conj} Suppose $\{S_i| i\in I\}$ is a locally finite
square tiling of the complex plane $\C$ such that each square
intersects exactly six others. Then all squares $S_i$ have the
same size.
\end{conj}

\end{document}